\documentclass{amsart}[12pt]
\usepackage{amssymb,amsmath}
\usepackage{enumerate}
\usepackage[english]{babel}
\usepackage{color}
\usepackage{comment}

\newtheorem{teo}{Theorem}[section]
\newtheorem{pro}{Proposition}[section]
\newtheorem{cor}{Corollary}[section]

\newtheorem{lm}{Lemma}[section]

\theoremstyle{definition}
\newtheorem{rem}{Remark}[section]
\newtheorem{df}{Definition}[section]

\newtheorem{ex}{Example}[section]

\title[Noncommutative modulated ergodic theorems]{$T$-admissible processes and Noncommutative Weighted Ergodic Theorems}
\keywords{Semifinite von Neumann algebra, noncommutative ergodic theorems, $T$-admissible processes, weighted ergodic theorems}
\subjclass[2020]{47A35, 46L51}

\author{Morgan O'Brien}
\address{Minnesota State University, Mankato\\ Department of Mathematics and Statistics \\ Mankato, MN 56001, USA}
\email{morgan.obrien.2@mnsu.edu, obrienmorganc@gmail.com}

\begin{document}
\begin{abstract}
	In this article, we study the bilaterally almost uniform (b.a.u.) convergence of weighted averages of a positive Dunford-Schwartz operator $T$ acting on the noncommutative $L_p$-spaces $L_p(\mathcal{M},\tau)$ associated to a semifinite von Neumann algebra $(\mathcal{M},\tau)$ for every $1\leq p<\infty$ for a variety of sequences which includes those generated by bounded i.i.d. sequences, the M\"{o}bius function, the Liouville function, and totally balanced automatic sequences. If $1<p<\infty$, then sequences generated by i.i.d. sequences with finite $q$-th moment for $q=\frac{p}{p-1}$ being the H\"{o}lder conjugate of $p$ are also shown to generate good weights for $L_p(\mathcal{M},\tau)$.
	
	We also study similar problems for $T$-admissible processes on a semifinite von Neumann algebra, showing that if a Wiener-Wintner type ergodic theorem holds for a class $\mathcal{W}\subset W_q$ of weights, then it also holds for strongly $p$-bounded $T$-admissible processes assuming that the duality $\frac{1}{p}+\frac{1}{q}=1$ holds and that $T$ is a normal $\tau$-preserving $*$-automorphism. A version of this result for arbitrary positive Dunford-Schwartz operators is also considered.
\end{abstract}
\date{October 25, 2026}
\maketitle

\begin{center}\textit{Dedicated to the memory of Professor Do\u{g}an \c{C}\"{o}mez.}\end{center}

\section{Introduction}\label{s1}

The study of weighted ergodic averages is a frequently studied subject in ergodic theory. For a fixed sequence $\alpha=(\alpha_k)_{k=0}^{\infty}\subset\mathbb{C}$ consider the weighted averages
\begin{align}\label{StandardAverages}\Big(M_n^{\alpha}(T_\phi)(f)\Big)(\omega):=\frac{1}{n}\sum_{k=0}^{n-1}\alpha_kf(\phi^k(\omega)),\end{align} where $(\Omega,\mathcal{F},\nu)$ is a probability space, $\omega\in\Omega$, $f\in L_1(\Omega,\nu)$, and $\phi:\Omega\to\Omega$ is a measure-preserving transformation.

In classical ergodic theory it is common to study which choices of sequence $\alpha$ guarantees the convergence of the averages (\ref{StandardAverages}) above for every probability space $(\Omega,\mathcal{F},\nu)$, measure-preserving transformation $\phi:\Omega\to\Omega$, $f\in L_1(\Omega,\nu)$, and $\nu$-a.e. $\omega\in\Omega$. If a sequence $\alpha$ manages to satisfy this, then it is said to be \textit{good for the individual ergodic theorem on $L_1$}. Some examples of good sequences for the individual ergodic theorem on $L_1$ are:
\begin{itemize}
\item[(i)] a linear power of a unit circle element, i.e. $\alpha_k=\lambda^{k}$ for a fixed $\lambda\in\mathbb{T}$ (equivalently, $\alpha_k=e^{2\pi itk}$ for some fixed $t\in\mathbb{R}$) \cite{ww41},
\item[(ii)] a bounded Besicovitch sequence \cite{rn},
\item[(iii)] a polynomial power of a unit circle element, i.e. $\alpha_k=e^{2\pi iP(k)}$ for a fixed real polynomial $P\in\mathbb{R}[X]$ \cite[Theorem 1]{lesigne},
\item[(iv)] a Hartman almost-periodic sequence with correlation and discrete spectral measure \cite[Theorem 5.2]{bl},
\item[(v)] a dynamically generated sequence, i.e. $\alpha_k=g(\psi^k(z))$, where $(Z,\mathcal{G},\rho)$ is a probability space, $\psi:Z\to Z$ is a measure-preserving transformation, $g\in L_\infty(Z,\rho)$, and $z\in Z$ \cite{bourgain88,bfko89},
\item[(vi)] one generated by a bounded independent identically distributed (i.i.d.) sequence, i.e. $\alpha_k=Y_k(z)$ with $(Y_k)_{k=0}^{\infty}$ being a bounded i.i.d. sequence on a probability space $(Z,\mathcal{G},\rho)$ and $z\in Z$ \cite[Theorem 2]{assani1} (see Subsection \ref{Applications} below for a precise definition),
\item[(vii)] $(\mu(k))_{k=0}^{\infty}$, where $\mu:\mathbb{N}_0\to\{-1,0,1\}$ is the M\"{o}bius function defined by $$\mu(n):=\begin{cases}1&\text{ if }n=1,\\(-1)^{k},&\text{ if }n\text{ has }k\text{ distinct prime factors},\\ 0&\text{ otherwise,}\end{cases}$$\cite[Proposition 3.1]{eakpldlr}, 
\item[(viii)] $(\lambda(k))_{k=0}^{\infty}$, where $\lambda:\mathbb{N}_0\to\{-1,1\}$ denotes the Lioville function defined by $$\lambda(n):=(-1)^{a_1+\cdots+a_k},\text{ where }n=p_1^{a_1}\cdots p_k^{a_k}\text{ is the prime factorization of }n,$$ \cite[Lemma 1]{BatemanChowla}, \cite{cw},
\item[(ix)] a totally balanced automatic sequence or an invertible automatic sequence \cite[Theorems B and C]{ek}, 
\item[(x)] a $q$-multiplicative sequence \cite{LeMa}, \cite{LeMaMo}, and \cite[Theorems 5 and 6]{fan},
\end{itemize}

The Wiener-Wintner ergodic theorem states that if $(\Omega,\mathcal{F},\mu)$ is a probability space, $\phi:\Omega\to\Omega$ is a measure-preserving transformation, and if $f\in L_1(\Omega,\mu)$, then there exists a set $\Omega_{f}\subseteq\Omega$ (independent of any weight $(\alpha_k)_{k=0}^{\infty}$) such that for every $\omega\in\Omega_{f}$ the averages in (\ref{StandardAverages}) above converge as $n\to\infty$ with $\alpha_k=\lambda^k$ for every $\lambda\in\mathbb{T}$. If a similar type of result holds for every weight $\alpha=(\alpha_k)_{k=0}^{\infty}$ in a family $\mathcal{W}$ of weights, then it is said that a Wiener-Wintner type result holds for $\mathcal{W}$. It is known that each of the families of sequences in (i) through (vi) satisfy a Wiener-Wintner type result, as well as (xi) through (xiii) below (with appropriate modifications).

If one fixes $p>1$ and considers the results above for every $f\in L_p(\Omega,\mu)$ instead of $L_1(\Omega,\mu)$, and if $1<q<\infty$ is such that $\frac{1}{p}+\frac{1}{q}=1$, then one can extend this list to include:
\begin{itemize}
\item[(xi)] $q$-Besicovitch sequences \cite[Theorem 3.5]{lot},
\item[(xii)] $g\in L_q(Z,\rho)$ instead of $g\in L_\infty(Z,\nu)$ in (v) above,
\item[(xiii)] $(Y_k)_{k=0}^{\infty}$ being an i.i.d. sequence with $\mathbb{E}(|Y_0|^q)<\infty$ instead of being bounded in (vi) above,
\item[(xiv)] $\alpha_k=\Lambda(k)$, where $\Lambda$ denotes the von Mangoldt function \cite{wie88}.
\end{itemize}
Such sequences are said to be \textit{good for the individual ergodic theorem on $L_p$}. 
The condition $\frac{1}{p}+\frac{1}{q}=1$ usually comes from usage of H\"{o}lder's inequality and is not always needed (such results are usually referred to as \textit{breaking duality}; see \cite[Theorem 1.12]{demeterjones} for (xi), and \cite{dltt} for (xii), and \cite{assani1} for (xiii) for a few examples). A class of Besicovitch sequences not satisfying the usual duality assumptions is known to be good for the individual ergodic theorem on $L_p$ as shown by Demeter and Jones in \cite[Theorem 1.12]{demeterjones}. An important fact about the sequences in this second list of weights is that they are not necessarily bounded, which can add some complexities to their proofs.

In Baxter and Olsen \cite[Theorem 2.19]{bo} it was shown that $f(\phi^k(\omega))$ in Equation (\ref{StandardAverages}) above can be replaced by $(T^k(f))(\omega)$, where $T:L_1(\Omega,\nu)\to L_1(\Omega,\nu)$ is a Dunford-Schwartz operator, i.e. a linear operator such that $\|T(f)\|_p\leq\|f\|_p$ for every $1\leq p\leq\infty$ and $f\in L_p(\Omega,\nu)$. \c{C}\"{o}mez, Lin, and Olsen expanded the work in \cite{comlio} to also shown that an analogous replacement can be made in (v) and (xii).

Another often studied problem in ergodic theory concerns the convergence of subadditive and superadditive processes. The weighted versions of these can be difficult to study directly (or even define properly), so we will focus on the class of $T$-admissible processes. These replace the ($T$-additive) sequence $(T^k(f))_{k=0}^{\infty}$ with a family $(f_k)_{k=0}^{\infty}$ of real-valued functions on $\Omega$ such that $T(f_k)\leq f_{k+1}$ for every $k\in\mathbb{N}_0$; call a sequence satisfying this latter condition $T$-admissible. These generate $T$-superadditive sequences by letting $F_n(\omega)=\sum_{k=0}^{n}f_k(\omega)$ and $F=(F_n)_{n=0}^{\infty}$ (meaning that $F_{n+m}\geq F_n+T^n(F_m)$ for every $m,n\geq0$) for $\omega\in\Omega$. One study of the weighted averages of these (including a Wiener-Wintner type result) was done by \c{C}\"{o}mez and Litvinov in \cite{comli2} when $T$ is the Koopman operator of a measure-preserving transformation of a probability space. Furthermore, it was proven that strongly $p$-bounded $T$-admissible sequences generate good weights for the individual ergodic theorem (similar to cases (v) and (xii) above), and conditions for which such sequences are bounded Besicovitch and $q$-Besicovitch sequences are found in \cite[Theorem 3.4]{comli2}.

One application of this type of result can be found in a remark in \cite{comli2} (at the end of Section 3 of that paper) to recurrence properties of dynamical systems. There they consider the return times of sets that are allowed to grow as $n$ tends to infinity (compared to, say, Poincar\'{e}'s recurrence theorem which has the same set for every time). In short, if $(\Omega,\mathcal{F},\nu)$ is a probability space and $\phi:\Omega\to\Omega$ is a measure-preserving transformation, then write $T_\phi(f):=f\circ\phi$ for the associated Koopman operator. If one considers a family $(A_k)_{k=0}^{\infty}\subset\mathcal{F}$ such that $A_k\subseteq A_{k+1}$ for every $k\in\mathbb{N}_0$ then the sequence $(\chi_{\phi^{-k}(A_k)})_{k=0}^{\infty}$ is a strongly $\infty$-bounded $T_{\phi}$-admissible sequence, which would imply that the averages
$$\frac{1}{n}\sum_{k=0}^{n-1}\chi_{A_k}(\phi^k(\omega))$$ converge as $n\to\infty$ for $\nu$-a.e. $\omega\in\Omega$. If the transformation $\phi$ is ergodic then the limit function is the $\nu$-a.e. constant function taking the value $\nu\left(\bigcup_{k=0}^{\infty}A_k\right)$.

In the noncommutative setting the list of weights that are good for the individual ergodic theorem is much shorter than the one given above. For the unweighted averages $\alpha=(1)_{k=0}^{\infty}$ the first result was shown by Yeadon in \cite[Theorem 1]{ye} for $L_1(\mathcal{M},\tau)$, which was later extended to $L_p(\mathcal{M},\tau)$ with $1\leq p<\infty$ by Junge and Xu in \cite[Corollary 6.4]{jx} (with a.u. convergence when $p\geq2$). In \cite[Lemma 4.2 and Theorem 4.6]{cls} Chilin, Litvinov, and Skalski generalized (i) and (ii) in the list above to hold in the noncommutative $L_1$-space associated to a semifinite von Neumann algebra $(\mathcal{M},\tau)$ under the additional assumption that the von Neumann algebra $\mathcal{M}$ has a separable predual (which is equivalent to $\mathcal{M}$ having an action on a separable Hilbert space). This was was later extended to every other $L_p(\mathcal{M},\tau)$ for $1<p<\infty$ by Chilin and Litvinov in \cite[Theorem 3.5]{cl1}. The author showed in \cite[Corollary 3.10]{ob4} that case (xi) holds in the noncommutative setting by proving an appropriate maximal inequality for unbounded weights. Litvinov proved that a Wiener-Wintner type result holds for (i) if $T$ is an ergodic normal $\tau$-preserving $T$-homomorphism in \cite[Theorem 5.1]{li2}. Hong and Sun extended this in \cite[Theorem 1.3]{hs} by proving that case (v) holds in a Wiener-Wintner type form for normal $\tau$-preserving $*$-automorphisms (where a multiparameter version of the result was considered). Recently Hong and Wong proved that (xiv) holds when $T$ is a $\tau$-preserving $*$-automorphism in \cite{hw}.

Outside of the list, \c{C}\"{o}mez and Litvinov showed in \cite[Theorem 5.3]{comli3} that one can consider a weight sequence consisting of operators instead of scalars. 
The author showed in \cite{ob2} that a large number of Hartman sequences satisfy a Wiener-Wintner type result for a special class of positive Dunford-Schwartz operators which satisfy a certain convergence condition on their iterates (such a condition holds when the restriction of $T$ to $L_2(\mathcal{M},\tau)$ self-adjoint); this class includes those whose restriction to $L_2$ are self-adjoint. 
The Hartman sequences considered include all of the cases (i) through (xiv) in the list above in some capacity, with bounded sequences extending to holding on $L_1(\mathcal{M},\tau)$ (where the methods mentioned do not necessarily hold since it is not a reflexive Banach space).

Some results relating to the b.a.u convergence of $T$-subadditive sequences were found by \cite[Theorem 1]{ja}. The only result regarding $T$-admissible sequences in particular in the noncommutative setting that the author is aware of was proven by \c{C}\"{o}mez and Litvinov in \cite[Theorem 3.5]{comli1}, where the $L_p$-norm convergence of moving average sequences for $T$-admissible processes was shown to hold.

There are two main goals we wish to prove in this article relating to the previously mentioned types of problems.

The first main goal greatly expands the list of good weights for the individual ergodic theorem in the noncommutative setting. This is done primarily in Theorem \ref{ConvergenceCriteriaI}, where it is shown that if a bounded weighting sequence $\alpha=(\alpha_k)_{k=0}^{\infty}\subset\mathbb{C}$ satisfies a certain boundedness condition on the exponential averages. 
Specifically this is done by imposing a bound on
$$\sup_{\lambda\in\mathbb{T}}\left|\frac{1}{n}\sum_{k=0}^{n-1}\alpha_k\lambda^k\right|$$ in a way so 
that they would be summable along appropriate subsequences, from which it will follow that it is a good weight on $L_p(\mathcal{M},\tau)$ for every $1\leq p<\infty$, for every positive Dunford-Schwartz operator $T\in DS^+(\mathcal{M},\tau)$, and for every semifinite von Neumann algebra $(\mathcal{M},\tau)$. This is the same method as the classical setting used in proving cases (vi)-(x) and some special cases of (iii) and (v) from  the list above. For completeness we will write out many of the details in applying the procedure to bounded i.i.d. sequences and outline the procedure for others. We will write out the details which show that i.i.d. sequences with finite $q$-th moment are good weights on $L_p(\mathcal{M},\tau)$ as long as $1<p<\infty$ and $\frac{1}{p}+\frac{1}{q}=1$, adding (xiii) to the list of good weights for von Neumann algebras. 

The second goal concerns $T$-admissible sequences in the noncommutative setting. In particular, Theorem \ref{TAdmSequencesAreGood} shows that if $\alpha$ is a good weight for a normal $\tau$-preserving $*$-automorphism $T$ on $L_p(\mathcal{M},\tau)$, then it is a good weight for any strongly $p$-bounded $T$-admissible process as well as long as the standard H\"{o}lder duality holds. This result will be proven in a Wiener-Wintner form, which allows for multiple sequences to be considered at the same time. A version of this result for arbitrary positive Dunford-Schwartz operators is also considered in Theorem \ref{TAdmDS}.

\section{Preliminaries}\label{Preliminaries}

Throughout this article $\mathbb{N}$ will denote the set of natural numbers, $\mathbb{N}_0:=\mathbb{N}\cup\{0\}$, and $\mathbb{T}:=\{\lambda\in\mathbb{C}:|\lambda|=1\}$ will be the unit circle. We will write $R[X]$ for the set of polynomials over the variable $X$ with coefficients in $R\in\{\mathbb{R},\mathbb{C}\}$. If $t\in\mathbb{R}$ we will let $[t]$ denote the greatest integer less than or equal to $t$, i.e. $[t]\in\mathbb{Z}$ is the unique integer such that $[t]\leq t<[t]+1$. If needed we will write $\frac{0}{0}:=0$.

For the rest of this article $\mathcal{H}$ will denote a fixed complex Hilbert space. Write $\textbf{1}:\mathcal{H}\to\mathcal{H}$ for the identity operator $\textbf{1}(\xi)=\xi$ for all $\xi\in\mathcal{H}$. If $\mathcal{X}$ is a complex Banach space and $T:\mathcal{X}\to\mathcal{X}$ is a bounded linear operator, then we will write $\|T\|_{\mathcal{X}\to\mathcal{X}}$ for its operator norm. For the special case where $\mathcal{X}=\mathcal{H}$ we will instead write $\|\cdot\|_{\infty}:=\|\cdot\|_{\mathcal{H}\to\mathcal{H}}$.

Call the pair $(\mathcal{M},\tau)$ a \textit{semifinite von Neumann algebra} if $\mathcal{M}$ is a von Neumann algebra acting on $\mathcal{H}$ and if $\tau$ is a normal semifinite faithful trace on $\mathcal{M}$. Let $\mathcal{P(M)}$ denote the set of all projections in $\mathcal{M}$, and write $e^{\perp}:=\textbf{1}-e$ for every $e\in\mathcal{P(M)}$.

An operator $x:\mathcal{D}_x\to\mathcal{H}$, where $\mathcal{D}_{x}\subset\mathcal{H}$, is affiliated with $\mathcal{M}$ if $yx\subseteq xy$ for every $y$ in the commutant $\mathcal{M}^{\prime}$ of $\mathcal{M}$. For such $x$, it is called $\tau$-measurable if for every $\epsilon>0$ there exists $e\in\mathcal{P(M)}$ such that $\tau(e^\perp)\leq\epsilon$ and $xe\in\mathcal{M}$.

Write $L_0(\mathcal{M},\tau)$ for the set of all $\tau$-measurable operators affiliated with $\mathcal{M}$. The \textit{measure topology} is defined on $L_0(\mathcal{M},\tau)$ by the following sets of neighborhoods of the constant zero operator $0$:
$$V(\epsilon,\delta):=\{x\in L_0(\mathcal{M},\tau):\|xe\|_\infty\leq\delta\text{ for some }e\in\mathcal{P(M)}\text{ with }\tau(e^\perp)\leq\epsilon\}.$$ $L_0(\mathcal{M},\tau)$ equipped with this topology and the closed sum, closed product, and adjoint operator turn it into a completely metrizable topological $*$-algebra. See \cite{ne} for more details.

If $(x_n)_{n=1}^{\infty}\subset L_0(\mathcal{M},\tau)$ and $x\in L_0(\mathcal{M},\tau)$, then $x_n\to x$ \textit{bilaterally almost uniformly} (abbreviated \textit{b.a.u.}) as $n\to\infty$ if for every $\epsilon>0$ there exists $e\in\mathcal{P(M)}$ such that $\tau(e^\perp)\leq\epsilon$ and $\|e(x_n-x)e\|_\infty\to0$ as $n\to\infty$. If instead $\|(x_n-x)e\|_\infty\to0$ as $n\to\infty$ then say that $x_n\to x$ \textit{almost uniformly} (abbreviated \textit{a.u.}). One can see that a.u. convergence implies b.a.u. convergence, but the converse is not true in general. It is known that both b.a.u. and a.u. convergence implies convergence in $L_0(\mathcal{M},\tau)$ with respect to the measure topology (see \cite{cls}).

If $x\in L_0(\mathcal{M},\tau)$, then write $x\geq0$ if $\langle x(\xi),\xi\rangle\geq0$ for every $\xi\in\mathcal{D}_{x}$. If $E\subseteq L_0(\mathcal{M},\tau)$, then write $E^{+}:=\{x\in E:x\geq0\}$.

If $x\in L_0(\mathcal{M},\tau)$, then it has a polar decomposition $x=u|x|$, where $u\in\mathcal{M}$ satisfies $\|u\|_\infty\leq1$ and $|x|\in L_0(\mathcal{M},\tau)^+$ satisfies $|x|^2=x^*x$. If $x\in L_0(\mathcal{M},\tau)^+$, then we may write it in its spectral decomposition as $x=\int_{[0,\infty)}\lambda de_{\lambda}$. We can extend $\tau$ from $\mathcal{M}^{+}$ to $L_0(\mathcal{M},\tau)^+$ by letting
$$\tau(x):=\sup_{n\in\mathbb{N}}\tau\left(\int_{[0,n]}\lambda de_{\lambda}\right).$$

If $1\leq p<\infty$ then for every $x\in L_0(\mathcal{M},\tau)$ define $\|x\|_p:=\tau(|x|^p)^{1/p}$. Let $L_p(\mathcal{M},\tau):=\{x\in L_0(\mathcal{M},\tau):\|x\|_p<\infty\}$. Write $L_\infty(\mathcal{M},\tau)=\mathcal{M}$. Then $L_p(\mathcal{M},\tau)$ is a Banach space with the norm $\|\cdot\|_p$ for every $1\leq p\leq\infty$. Whenever it won't cause confusion we may write $L_p$ instead of $L_p(\mathcal{M},\tau)$ if $p=0$ or $1\leq p\leq\infty$. It is known that the inclusion maps $L_p\hookrightarrow L_0$ are continuous for $1\leq p<\infty$, so that $L_p$-norm convergence implies convergence with respect to the measure topology on $L_0(\mathcal{M},\tau)$.

Let $T:L_1+\mathcal{M}\to L_1+\mathcal{M}$ be a linear operator. Then $T$ is called a \textit{Dunford-Schwartz operator} if $\|T(x)\|_p\leq\|x\|_p$ for every $1\leq p\leq\infty$ and $x\in L_p(\mathcal{M},\tau)$. The operator is called \textit{positive} if $T(x)\geq0$ whenever $x\geq0$. Write $DS^{+}(\mathcal{M},\tau)$ for the set of all positive Dunford-Schwartz operators.

Fix $1\leq q\leq\infty$ and $\alpha=(\alpha_k)_{k=0}^{\infty}$. If $q<\infty$ then write $$\|\alpha\|_{W_q}:=\left(\limsup_{n\to\infty}\frac{1}{n}\sum_{k=0}^{n-1}|\alpha_k|^q\right)^{1/q}$$ and $|\alpha|_{W_q}$ for the same quantity with $\limsup_{n\to\infty}$ replaced by $\sup_{n\in\mathbb{N}}$. If $q=\infty$ then write $\|\alpha\|_{W_\infty}=|\alpha|_{W_\infty}:=\sup_{k\in\mathbb{N}_0}|\alpha_k|$. Let $$W_q:=\{\alpha=(\alpha_k)_{k=0}^{\infty}\subset\mathbb{C}:\|\alpha\|_{W_q}<\infty\}.$$ Then $W_q$ is a vector space and $\|\cdot\|_{W_q}$ is a seminorm on it for every $1\leq q\leq\infty$. One can find that $\|\alpha\|_{W_q}<\infty$ if and only if $|\alpha|_{W_q}<\infty$.

Given $1\leq p,q\leq\infty$, $T\in DS^+(\mathcal{M},\tau)$, $x\in L_p(\mathcal{M},\tau)$, and $\alpha=(\alpha_k)_{k=0}^{\infty}\in W_q$, then for every $t\geq1$ write the weighted average of $x$ under $T$ weighted by $\alpha$ as
$$M_t^{\alpha}(T)(x):=\frac{1}{t}\sum_{k=0}^{[t]-1}\alpha_k T^k(x).$$
Our focus will be on the convergence of these averages for $t\in\mathbb{N}$, in which case we will write $(M_n^{\alpha}(T)(x))_{n=1}^{\infty}$ with $M_n^{\alpha}(T)(x)=\frac{1}{n}\sum_{k=0}^{n-1}\alpha_kT^k(x)$. The extension for all other values of $t\in[0,\infty)\setminus\mathbb{N}$ is used for notational convenience later on and for consistency with particular results in the literature. If $\alpha=(1)_{k=0}^{\infty}$ then we will write $M_t(T)$ for the unweighted averages of $T$ instead of $M_t^{\alpha}(T)$ for every $t\geq1$.

If $1\leq p,q\leq\infty$, $T\in DS^+(\mathcal{M},\tau)$, and $\mathcal{W}\subset W_q$ are fixed, then for $x\in L_p(\mathcal{M},\tau)$ write $x\in bWW_p^{T}(\mathcal{W})$ if for every $\epsilon>0$ there exists $e\in\mathcal{P(M)}$ such that $\tau(e^\perp)\leq\epsilon$ and
$$(eM_n^{\alpha}(T)(x)e)_{n=1}^{\infty}\text{ converges with respect to }\|\cdot\|_\infty\text{ in }\mathcal{M}\text{ for every }\alpha\in\mathcal{W}.$$
In other words, if $x\in bWW_p^{T}(\mathcal{W})$ then $(M_n^{\alpha}(T)(x))_{n=1}^{\infty}$ converges b.a.u. as $n\to\infty$ for every $\alpha\in\mathcal{W}$, and for every $\epsilon>0$ the same projection can be used for each $\alpha\in\mathcal{W}$. The formulation of the b.a.u. convergence of weighted ergodic averages can be written in this notation by using $\mathcal{W}=\{\alpha\}$ for the relevant weight $\alpha\in W_q$.

A few important results from the literature that we will need to use in this article can be summarized as follows.

\begin{teo}\label{NonCommDunSchErgThm}(Cf. \cite[Theorem 1]{ye}, \cite[Corollary 6.4]{jx})
Let $(\mathcal{M},\tau)$ be a semifinite von Neumann algebra, $T\in DS^+(\mathcal{M},\tau)$, and $1\leq p<\infty$. If $x\in L_p(\mathcal{M},\tau)$, then the (unweighted) averages $(M_n(T)(x))_{n=1}^{\infty}$ converge b.a.u. to some operator $F_T(x)\in L_p(\mathcal{M},\tau)$ as $n\to\infty$. This convergence occurs a.u. if $p\geq2$.
\end{teo}

\begin{teo}\label{WeightedErgodicTheoryResults}(Cf. \cite[Theorem 3.4]{ob2}, \cite[Corollary 3.5]{ob4})
Let $(\mathcal{M},\tau)$ be a semifinite von Neumann algebra, $T\in DS^{+}(\mathcal{M},\tau)$, and $1\leq p,q\leq\infty$ satisfy either $1<p,q<\infty$ and $\frac{1}{p}+\frac{1}{q}=1$ or $p=1$ and $q=\infty$. If $x\in L_p(\mathcal{M},\tau)$, then there exists a constant $C_p>0$ (depending only on $p$) such that for every $\lambda>0$ there exists $e\in\mathcal{P(M)}$ with $$\tau(e^\perp)\leq\left(\frac{C_p\|x\|_p}{\lambda}\right)^{p} \ \text{ and } \ \sup_{(n,\alpha)\in\mathbb{N}\times W_q}\left\|e\left(\frac{1}{|\alpha|_{W_q}}M_n^{\alpha}(T)(x)\right)e\right\|_\infty\leq\lambda.$$
Consequently, if $\mathcal{W}\subseteq W_q$ then $bWW_p^{T}(\mathcal{W})$ is a closed subspace of $L_p(\mathcal{M},\tau)$.
\end{teo}

\section{Noncommutative Weighted Ergodic Theorems}\label{NCWET}

\subsection{General Method}\label{GeneralMethod}

In this subsection we will prove a general method that one can use to prove a variety of weighted ergodic theorems in the von Neumann algebra setting. This is the method used in classical ergodic theory which was used to prove that i.i.d. sequences and the M\"{o}bius function generate good weights, among others mentioned in the introduction. 

\begin{lm}\label{DecreasingToZeroImpliesBAUConvergent}
	Let $(\mathcal{M},\tau)$ be a semifinite von Neumann algebra. If $(x_n)_{n=0}^{\infty}\subset L_0(\mathcal{M},\tau)^{+}$ is a decreasing sequence (with respect to the ordering $\leq$ on $L_0(\mathcal{M},\tau)$) such that $x_n\to0$ in measure, then $x_n\to0$ b.a.u. as $n\to\infty$.
\end{lm}
\begin{proof}
	By \cite[Proposition 1]{gl} it is known that since the sequence converges to $0$ in measure there exists a subsequence $(x_{n_k})_{k=0}^{\infty}$ of the original sequence $(x_n)_{n=0}^{\infty}$ such that $x_{n_k}\to0$ a.u., and so b.a.u., as $k\to\infty$.
	
	Assume $\epsilon,\delta>0$, and let $e\in\mathcal{P(M)}$ be such that $\tau(e^\perp)\leq\epsilon$ and $\|ex_{n_k}e\|_\infty\to0$ as $k\to\infty$. Let $K\in\mathbb{N}_0$ be such that $k\geq K$ implies $\|ex_{n_k}e\|_\infty\leq\delta$. If $n\geq n_K$, then it follows that $x_n\leq x_{n_K}$ in $L_0(\mathcal{M},\tau)^+$ since $(x_n)_{n=0}^{\infty}$ is decreasing. Therefore it follows that $0\leq ex_ne\leq ex_{n_K}e$ in $L_0(\mathcal{M},\tau)^{+}$ as well; furthermore, the projection $e$ can be chosen so that $x_ne\in\mathcal{M}$ for each $n\in\mathbb{N}_0$, so that the inequality actually holds in $\mathcal{M}^{+}$. Hence, $$0\leq\|ex_ne\|_\infty\leq\|ex_{n_K}e\|_\infty\leq\delta$$ by properties of the supremum norm and the fact that $z\leq\|z\|_\infty\textbf{1}$ for $z\in\mathcal{M}^+$. Since $n\geq n_K$ and $\delta>0$ were arbitrary, it follows that $\|ex_ne\|_\infty\to0$ as $n\to\infty$, and so $x_n\to0$ b.a.u. as $n\to\infty$ since $\epsilon>0$ was arbitrary as well.
\end{proof}

\begin{lm}\label{SummableImpliesBAUconvergent}
	Let $(\mathcal{M},\tau)$ be a semifinite von Neumann algebra and fix $1\leq p<\infty$. If $(x_n)_{n=0}^{\infty}\subset L_p(\mathcal{M},\tau)^{+}$ is such that $\sum_{k=0}^{\infty}x_k\in L_p(\mathcal{M},\tau)^+$ (with convergence being with respect to the $\|\cdot\|_p$-norm), then $x_n\to0$ b.a.u. as $n\to\infty$.
\end{lm}
\begin{proof}
	If $y_n:=\sum_{k=0}^{n}x_k\in L_p^{+}$ for every $n\in\mathbb{N}_0$ and $y:=\sum_{k=0}^{\infty}x_k\in L_p^+$, then by assumption it follows that $\|y-y_n\|_p\to0$ as $n\to\infty$. Since $y_n\geq0$ for every $n\in\mathbb{N}_0$ it follows that $y\in L_0^+$ as well since $L_0^+$ is closed with respect to the measure topology; furthermore, since $(y_n)_{n=0}^{\infty}$ is an increasing sequence one can find that $y_n\leq y$ and $y-y_{n+1}\leq y-y_{n}$ for every $n\in\mathbb{N}_0$. Since convergence with respect to the norm $\|\cdot\|_p$ of $L_p(\mathcal{M},\tau)$ implies convergence in measure, it follows that $y-y_n\to0$ in measure. Therefore $(y-y_n)_{n=0}^{\infty}$ is a decreasing sequence in $L_0^+$ which converges to $0$ in measure, hence $y-y_n\to0$ b.a.u. as $n\to\infty$ by Lemma \ref{DecreasingToZeroImpliesBAUConvergent}.
	
	Assume $\epsilon>0$ and let $e\in\mathcal{P(M)}$ be such that $\tau(e^\perp)\leq\epsilon$ and $\|e(y-y_n)e\|_\infty\to0$ as $n\to\infty$. Since $x_n=y_n-y_{n-1}$ for every $n\in\mathbb{N}_0$, we find that
	$$\|ex_ne\|_\infty\leq\|e(y_n-y)e\|_\infty+\|e(y-y_{n-1})e\|_\infty\to0\text{ as }n\to\infty.$$ Therefore $\|ex_ne\|_\infty\to0$ as $n\to\infty$, and since $\epsilon>0$ was arbitrary it follows that $x_n\to0$ b.a.u. as $n\to\infty$.
\end{proof}

The following lemma, a version of the ``subsequence argument'' which allows one to prove the convergence of averages by proving them along lacunary subsequences, will play key role in the proofs of our results that follow. It is also the main spot where issues relating to noncommutativity arise since the standard proof of this result has one apply $\limsup_{m\to\infty}$ to each term in a chain of inequalities of nonnegative numbers, but taking the $\limsup$ of a sequence of positive operators does not necessarily make sense in the von Neumann algebra setting. Fortunately, properties of the ordering $\leq$ on $\mathcal{M}^{+}$ and $L_0^{+}$ and the norm of $\mathcal{M}$ will allow a similar argument to work with a few careful modifications, as will be seen in the proof.

The classical version of this result often states that a certain assumption holds for all $\rho>1$, but its proof actually only needs this assumption to hold for each element of particularly chosen sequences in $(1,\infty)$ that decrease to $1$. Using this later condition (instead of ``for all $\rho>1$'') is necessary later for the proof that i.i.d. sequences are good weights for the individual ergodic theorem which uses the version stated here to guarantee the set of $\omega\in\Omega$ for which the result holds is measurable.

\begin{lm}\label{SubsequenceArgument}
	Let $(x_k)_{k=0}^{\infty}\subset L_0(\mathcal{M},\tau)^{+}$. For every $t\geq1$ write
	$$A_t:=\frac{1}{t}\sum_{k=0}^{[t]-1}x_k.$$
	
	Let $(\rho_j)_{j=0}^{\infty}\subset(1,\infty)$ be a sequence such that $\rho_j\to1$ as $j\to\infty$ and such that $\{\rho_j^m:m\in\mathbb{N}\}\subseteq\{\rho_\ell^m:m\in\mathbb{N}\}$ whenever $j\leq\ell$. Assume further that for every $j\in\mathbb{N}_0$ there exists $y_{j}\in L_0(\mathcal{M},\tau)^{+}$ such that $A_{\rho_j^m}\to y_{j}$ b.a.u. as $m\to\infty$.
	
	Then $y_{j}=y_{\ell}$ for every $j,\ell\in\mathbb{N}_0$, and if $y$ denotes this common limit then the full sequence $(A_n)_{n=1}^{\infty}$ converges to $y$ b.a.u. as $n\to\infty$.
\end{lm}
\begin{proof}
	Observe that $(A_{\rho_{\ell}^m})_{m=1}^{\infty}$ is a subsequence of $(A_{\rho_{j}^m})_{m=1}^{\infty}$ whenever $j\leq\ell$ by the subset condition on $(\rho_j)_{j=0}^{\infty}$ from the assumption. Since b.a.u. convergence implies convergence in measure and since $L_0(\mathcal{M},\tau)$ is Hausdorff with the measure topology, it follows that $y_{j}=y_{\ell}$ for every $\ell,j\in\mathbb{N}_0$. Let $y$ denote this common b.a.u. limit in $L_0(\mathcal{M},\tau)^{+}$. 
	
	Assume $\epsilon>0$. Since $A_{\rho_j^m}\to y$ b.a.u. as $m\to\infty$ for every $j\in\mathbb{N}_0$ by assumption, for every such $j$ there exists $e_{j}\in\mathcal{P(M)}$ such that $\tau(e_{j}^{\perp})\leq\frac{\epsilon}{2^{j+1}}$, $e_{j}x_ke_{j}\in\mathcal{M}^{+}$ for every $k\in\mathbb{N}_0$, and such that $\|e_{j}(A_{\rho_j^{m}}-y)e_{j}\|_\infty\to0$ as $m\to\infty$. Let $e=\bigwedge_{j=0}^{\infty}e_j$, noting that $ex_ke\in\mathcal{M}^{+}$ for every $k\in\mathbb{N}_0$ and that $\tau(e^\perp)\leq\epsilon$ and $\|e(A_{\rho_j^{m}}-y)e\|_\infty\to0$ as $m\to\infty$ for every $j\in\mathbb{N}_0$.
	
	Fix $j\in\mathbb{N}_0$. If $t\geq\rho_j$, then let $m\in\mathbb{N}$ be such that $\rho_j^m\leq t<\rho_j^{m+1}$ (noting that $[\rho_j^m]\leq[t]\leq[\rho_j^{m+1}]$ as well). With respect to the ordering on $L_0^{+}$ we have that
	$$A_t=\frac{1}{t}\sum_{k=0}^{[t]-1}x_k
	\leq\frac{1}{\rho_j^m}\sum_{k=0}^{[\rho_j^{m+1}]-1}x_k
	=\rho_j\left(\frac{1}{\rho_j^{m+1}}\sum_{k=0}^{[\rho_j^{m+1}]-1}x_k\right)
	=\rho_j A_{\rho_j^{m+1}}$$
	and
	$$A_t=\frac{1}{t}\sum_{k=0}^{[t]-1}x_k
	\geq\frac{1}{\rho_j^{m+1}}\sum_{k=0}^{[\rho_j^m]-1}x_k
	=\frac{1}{\rho_j}\left(\frac{1}{\rho_j^m}\sum_{k=0}^{[\rho^m]-1}x_k\right)
	=\frac{1}{\rho_j}A_{\rho_j^m}.$$
	Hence we have that
	$$0\leq\frac{1}{\rho_j}A_{\rho_j^m}\leq A_t\leq \rho_j A_{\rho_j^{m+1}}$$ in $L_0^{+}$, which then implies that 
	\begin{align}\label{TheLIMSUPLemmaInequality}0\leq e\left(\frac{1}{\rho_j}A_{\rho_j^m}\right)e\leq eA_{t}e\leq e\left(\rho_j A_{\rho_j^{m+1}}\right)e\end{align} in $\mathcal{M}^+$. Observe that 
	\begin{align*}
		\left\|e\left(\rho_{j} A_{\rho_{j}^{m+1}}-\frac{1}{\rho_{j}}A_{\rho_{j}^m}\right)e\right\|_\infty
		\leq&\left\|e\left(\rho_{j} A_{\rho_{j}^{m+1}}-A_{\rho_{j}^{m+1}}\right)e\right\|_\infty \\
		&+\left\|e\left(A_{\rho_{j}^{m+1}}-A_{\rho_{j}^m}\right)e\right\|_\infty \\
		&+\left\|e\left(A_{\rho_{j}^{m}}-\frac{1}{\rho_{j}^m}A_{\rho_{j}^{m}}\right)e\right\|_\infty \\
		=&(\rho_{j}-1)\|eA_{\rho_{j}^{m+1}}e\|_\infty \\
		&+\|e(A_{\rho_{j}^{m+1}}-A_{\rho_{j}^m})e\|_\infty \\
		&+\left(1-\frac{1}{\rho_{j}}\right)\|eA_{\rho_{j}^m}e\|_\infty
	\end{align*}
	
	Since $\|e(A_{\rho_{j}^m}-y)e\|_\infty\to0$ as $m\to\infty$, it follows that $\|e(A_{\rho_{j}^{m+1}}-A_{\rho_{j}^m})e\|_\infty\to0$ and $\|eA_{\rho_{j}^m}e\|_\infty\to\|eye\|_\infty$ as well, both as $m\to\infty$. Applying $\limsup_{m\to\infty}$ to both sides of this inequality then shows that
	\begin{align*}
		\limsup_{m\to\infty}\left\|e\left(\rho_{j} A_{\rho_{j}^{m+1}}-\frac{1}{\rho_{j}}A_{\rho_{j}^m}\right)e\right\|_\infty
		\leq&\limsup_{m\to\infty}\left((\rho_{j}-1)\|eA_{\rho_{j}^{m+1}}e\|_\infty\right) \\
		& +\limsup_{m\to\infty}\left(\|e(A_{\rho_{j}^{m+1}}-A_{\rho_{j}^m})e\|_\infty\right) \\
		& +\limsup_{m\to\infty}\left(\left(1-\frac{1}{\rho_{j}}\right)\|eA_{\rho_{j}^m}e\|_\infty\right) \\
		=&(\rho_{j}-1)\|eye\|_\infty+0+\left(1-\frac{1}{\rho_{j}}\right)\|eye\|_\infty \\
		=&\left(\rho_{j}-\frac{1}{\rho_{j}}\right)\|eye\|_\infty.
	\end{align*}
	
	Fix $\delta>0$. From the $\limsup_{m\to\infty}$ above, $(eA_{\rho_{j}^m}e)_{m=1}^{\infty}$ being Cauchy with respect to $\|\cdot\|_\infty$, and Inequality (\ref{TheLIMSUPLemmaInequality}) we find that there exists $N_{j}\in\mathbb{N}$ such that $m,n\geq N_{j}$ implies
	$$\left\|e\left(\rho A_{\rho_{j}^{m+1}}-A_t\right)e\right\|_\infty
	\leq\left\|e\left(\rho_{j} A_{\rho_{j}^{m+1}}-\frac{1}{\rho_{j}}A_{\rho_{j}^m}\right)e\right\|_\infty
	\leq\left(\rho_{j}-\frac{1}{\rho_{j}}\right)\|eye\|_\infty+\frac{\delta}{12},$$
	$$\left\|e\left(A_t-\frac{1}{\rho_{j}}A_{\rho_{j}^m}\right)e\right\|_\infty
	\leq\left\|e\left(\rho_{j} A_{\rho_{j}^{m+1}}-\frac{1}{\rho_{j}}A_{\rho_{j}^m}\right)e\right\|_\infty
	\leq\left(\rho_{j}-\frac{1}{\rho_{j}}\right)\|eye\|_\infty+\frac{\delta}{12},$$
	$$\left\|e\left(A_{\rho_{j}^{m}}-A_{\rho_{j}^{n}}\right)e\right\|_\infty\leq\frac{\delta}{4\rho_{j}}.$$

	Therefore, if $s,t\geq \rho_{j}^{N_{j}}$ then there exists $m_1,m_2\in\mathbb{N}$ with $m_1,m_2\geq N_{j}$ such that $\rho_{j}^{m_1}\leq t<\rho_{j}^{m_{1}+1}$ and $\rho_{j}^{m_2}\leq s<\rho_{j}^{m_2+1}$, from which we find that
	\begin{align*}
		\|e(A_s-A_t)e\|_\infty
		\leq&\left\|e\left(A_s-\frac{1}{\rho_{j}}A_{\rho_{j}^{m_2}}\right)e\right\|_\infty+\left\|e\left(\frac{1}{\rho_{j}}A_{\rho_{j}^{m_2}}-\rho_{j} A_{\rho_{j}^{m_2}}\right)e\right\|_\infty \\
		&+\rho_j\|e(A_{\rho_{j}^{m_2}}- A_{\rho_{j}^{m_1}})e\|_\infty+\|e(\rho_{j} A_{\rho_{j}^{m_1}}- A_t)e\|_\infty \\
		\leq&3\left(\rho_{j}-\frac{1}{\rho_{j}}\right)\|eye\|_\infty+\frac{\delta}{2}
	\end{align*}

	Since $\rho_{j}\to1$ as $j\to\infty$ (and so $0<\rho_{j}-\frac{1}{\rho_{j}}\to0$), there exists $J\in\mathbb{N}_0$ such that $j\geq J$ implies $\rho_j-\frac{1}{\rho_j}\leq\frac{\delta}{6(\|eye\|_\infty+1)}$. Therefore if $m,n\in\mathbb{N}$ are such that $m,n\geq[\rho_J^{N_{J}}]+1\geq\rho_J^{N_{J}}$, then
	$$\|e(A_m-A_n)e\|_\infty
	\leq3\left(\rho_J-\frac{1}{\rho_J}\right)\|eye\|_\infty+\frac{\delta}{2}
	\leq\frac{3\delta\|eye\|_\infty}{6(\|eye\|_\infty+1)}+\frac{\delta}{2}
	\leq\delta.
	$$
	Since $m,n\geq[\rho_J^{N_{J}}]+1\in\mathbb{N}$ and $\delta>0$ were arbitrary, it follows that $(eA_ne)_{n=1}^{\infty}$ is Cauchy, and so convergent, in $\mathcal{M}$. Since $\epsilon>0$ was arbitrary, it follows that $(A_n)_{n=1}^{\infty}$ converges b.a.u. as $n\to\infty$. 
	
	Since a subsequence of $(A_{[\rho_0^m]})_{m=1}^{\infty}$ is a subsequence of $(A_n)_{n=1}^{\infty}$ (namely, one may need to remove a few terms from $(A_{[\rho_0^m]})_{m=1}^{\infty}$ in case $[\rho_{0}^m]=[\rho_{0}^{m+1}]$ for some $m\in\mathbb{N}$), and since $(A_{[\rho_0^m]})_{m=1}^{\infty}$ has the same b.a.u. limit as $(A_{\rho_0^m})_{m=1}^{\infty}$ since $$A_{[\rho_0^m]}=\frac{[\rho_0^m]}{\rho_0^m}A_{\rho_0^m}$$ and $\frac{[\rho_0^m]}{\rho_0^m}\to1$ as $m\to\infty$, it follows that the b.a.u. limit of $(A_n)_{n=1}^{\infty}$ is $y$ as well.
\end{proof}

\begin{rem}
As was stated before the proof, the assumption and conclusion\\

\textit{
``Let $(\rho_j)_{k=0}^{\infty}\subset(1,\infty)$ be a sequence such that $\rho_j\to1$ as $k\to\infty$ and such that $\{\rho_j^m:m\in\mathbb{N}\}\subseteq\{\rho_\ell^m:m\in\mathbb{N}\}$ whenever $j\leq\ell$. Assume further that for every $j\in\mathbb{N}_0$ there exists $y_{j}\in L_0(\mathcal{M},\tau)^{+}$ such that $A_{\rho_j^m}\to y_{j}$ b.a.u. as $m\to\infty$.''
} and ``Then $y_j=y_{\ell}$ for every $j,\ell\in\mathbb{N}_0$,''\\

\noindent in the previous result are usually written in the forms \\

\textit{``For every $\rho>1$ assume there exists $y_{\rho}\in L_0(\mathcal{M},\tau)$ such that $A_{\rho^m}\to y_{\rho}$ as $m\to\infty$.'' and ``If $y_{\rho}$ denotes the b.a.u. limit for each $\rho>1$, then $y_{\rho}=y_{\sigma}$ for every $\rho,\sigma>1$.''}\\

\noindent The latter implies the former, but not vice-versa. We use the stated version mostly for technical reasons relating to a step of the proof of a result regarding i.i.d. sequences later in this article. A similar modification can be made in Theorem \ref{ConvergenceCriteriaI} below as well.
\end{rem}

The following is one of the main results of the paper, which provides a general condition for one to check to determine whether b.a.u. convergence of weighted averages occurs to the von Neumann algebra setting. In particular, it allows one to conclude b.a.u. convergence of noncommutative weighted averages directly from some growth properties about the weights, which has been studied frequently in classical ergodic theory, which will be seen afterwards.

\begin{teo}\label{ConvergenceCriteriaI}
	Let $(\mathcal{M},\tau)$ be a semifinite von Neumann algebra, $T\in DS^+(\mathcal{M},\tau)$, $1\leq p<\infty$, and let $\alpha=(\alpha_k)_{k=0}^{\infty}\in W_\infty$. For every $t\geq1$ define the polynomial $\widehat{M}_{t}^{\alpha}\in\mathbb{C}[\lambda]$ by
	$$
	\widehat{M}_t^{\alpha}(\lambda):=\frac{1}{t}\sum_{k=0}^{[t]-1}\alpha_k\lambda^k, \ \lambda\in\mathbb{T}.$$

	Let $(\rho_j)_{j=0}^{\infty}\subset(1,\infty)$ be a sequence such that $\rho_j\to1$ as $j\to\infty$ and such that $\{\rho_j^m:m\in\mathbb{N}\}\subseteq\{\rho_\ell^m:m\in\mathbb{N}\}$ whenever $j\leq\ell$. Assume further that
	\begin{align}\label{ConvergenceCriteriaEquation}\sum_{m=1}^{\infty}\sup_{\lambda\in\mathbb{T}}\left|\widehat{M}_{\rho_j^m}^{\alpha}(\lambda)\right|^2<\infty\end{align} for every $j\in\mathbb{N}_0$.
	
	Then $(M_n^{\alpha}(T)(x))_{n=1}^{\infty}$ converges b.a.u. as $n\to\infty$ for every $x\in L_p(\mathcal{M},\tau)$, and the limit operator is $0$ if either $1<p<\infty$ or $p=1$ and $\tau$ is finite.
\end{teo}
\begin{proof}
	First assume that $\alpha_k\in\mathbb{R}$ for every $k\in\mathbb{N}_0$. 
	
	Fix $x\in L_2^+$. Since $T$ is a contraction of $L_2(\mathcal{M},\tau)$, by applying von Neumann's Inequality for Hilbert space contractions and the Maximum Modulus Principle one finds that
	$$\|M_t^{\alpha}(T)\|_{L_2\to L_2}
	\leq\sup_{\lambda\in\mathbb{T}}|\widehat{M}_t^{\alpha}(\lambda)|,
	$$ 
	for every $t\geq1$, where $\|\cdot\|_{L_2\to L_2}$ denotes the operator norm of an element of $\mathcal{B}(L_2(\mathcal{M},\tau))$.
	
	Fix $j\in\mathbb{N}_0$ and observe that the above implies
	$$\sum_{m=1}^{\infty}\|M_{\rho_j^m}^{\alpha}(T)(x)\|_2^2
	\leq\sum_{m=1}^{\infty}\|M_{\rho_j^m}^{\alpha}(T)\|_{L_2\to L_2}^2\|x\|_2^2
	\leq\|x\|_2^2\sum_{m=1}^{\infty}\sup_{\lambda\in\mathbb{T}}|\widehat{M}_{\rho_j^m}^{\alpha}(\lambda)|^2.$$
	Since the right-hand side of this inequality is finite by assumption, it follows that $\sum_{m=1}^{\infty}\|M_{\rho_j^m}^{\alpha}(T)(x)\|_2^2<\infty$ as well. Furthermore, since
	$$\sum_{m=1}^{\infty}\left\|M_{\rho_j^m}^{\alpha}(T)(x)\right\|_2^2
	=\sum_{m=1}^{\infty}\left\||M_{\rho_j^m}^{\alpha}(T)(x)|^2\right\|_1,$$ it follows that $\sum_{m=1}^{\infty}|M_{\rho_j^m}^{\alpha}(T)(x)|^2\in L_1^{+}$ (where the convergence is with respect to the norm $\|\cdot\|_1$) since $L_1$ is a Banach space, $L_0^+$ is closed in the measure topology, and $L_1^+\subset L_0^+$. Lemma \ref{SummableImpliesBAUconvergent} then implies that $|M_{\rho_j^m}^{\alpha}(T)(x)|^2\to0$ b.a.u. as $m\to\infty$.
	
	By the polar decomposition of operators in $L_0(\mathcal{M},\tau)$, for every $m\in\mathbb{N}_0$ there exists $u_m\in\mathcal{M}$ with $\|u_m\|_\infty\leq1$ such that $M_{\rho_j^m}^{\alpha}(T)(x)=u_m|M_{\rho_j^m}^{\alpha}(T)(x)|$. Assume $\epsilon>0$ and let $e\in\mathcal{P(M)}$ be such that $\tau(e^\perp)\leq\epsilon$ and $\|e|M_{\rho_j^m}^{\alpha}(T)(x)|^2e\|_\infty\to0$ as $n\to\infty$. Then
	\begin{align*}
		\left\|M_{\rho_j^m}^{\alpha}(T)(x)e\right\|_\infty 
		=&\left\|u_m|M_{\rho_j^m}^{\alpha}(T)(x)|e\right\|_\infty \\
		\leq&\|u_m\|_\infty\left\|\left(|M_{\rho_j^m}^{\alpha}(T)(x)|e\right)^*|M_{\rho^m}^{\alpha}(T)(x)|e\right\|_\infty^{1/2} \\
		\leq&\left\|e|M_{\rho_j^m}^{\alpha}(T)(x)|^2e\right\|_\infty^{1/2}\to0
	\end{align*} as $m\to\infty$. Since $\epsilon>0$ was arbitrary, it follows that $M_{\rho_j^m}^{\alpha}(T)(x)\to0$ a.u., and so b.a.u., as $m\to\infty$. 
	
	Since $\alpha\subset\mathbb{R}$ is bounded, it follows that $|\alpha|_{W_\infty}+\alpha_k\geq0$ for every $k\in\mathbb{N}_0$. Defining $x_k:=(|\alpha|_{W_\infty}+\alpha_k)T^k(x)\in L_2^{+}$ one finds that $$\frac{1}{\rho_j^m}\sum_{k=0}^{[\rho_j^m]-1}x_k=\frac{[\rho_j^m]}{\rho_j^m}\cdot|\alpha|_{W_\infty}M_{[\rho_j^m]}(T)(x)+M_{\rho_j^m}^{\alpha}(T)(x)$$
	for every $m\in\mathbb{N}$. Let $F_T(x)\in L_2^{+}$ denote the b.a.u. limit of $(M_n(T)(x))_{n=1}^{\infty}$, which exists by \cite[Theorem 6.3]{jx}. Then the first term in the above converges to $|\alpha|_{W_\infty}F_T(x)$ b.a.u. as $m\to\infty$ (since $\frac{[\rho_j^m]}{\rho_j^m}\to1$ as $m\to\infty$) and the second term converges to $0$ b.a.u. as $m\to\infty$ by the arguments above; hence, the left hand side converges to $|\alpha|_{W_\infty}F_T(x)$ b.a.u. as $m\to\infty$. 
	
	Since $j\in\mathbb{N}_0$ was arbitrary, it follows that $\frac{1}{n}\sum_{k=0}^{n}x_k\to|\alpha|_{W_\infty}F_T(x)$ b.a.u. as $n\to\infty$ by Lemma \ref{SubsequenceArgument}. 
	Since
	$$M_n^{\alpha}(T)(x)=\frac{1}{n}\sum_{k=0}^{n-1}x_k-|\alpha|_{W_\infty}M_n(T)(x)$$ for every $n\in\mathbb{N},$ and since both terms on the right side converge to $|\alpha|_{W_\infty}F_T(x)$ b.a.u. as $n\to\infty$, it follows that $M_n^{\alpha}(T)(x)\to|\alpha|_{W_\infty}F_T(x)-|\alpha|_{W_\infty}F_T(x)=0$ b.a.u. as $n\to\infty$. Since $x\in L_2^+$ was arbitrary and $\text{span}(L_2^{+})=L_2$, it follows that the b.a.u. convergence to $0$ occurs for every $x\in L_2$.
	
	To extend the convergence to $L_p(\mathcal{M},\tau)$ for any other $1\leq p<\infty$ one can apply \cite[Theorem 2.1]{cl1} and \cite[Theorem 2.1]{ob2} since $L_1\cap\mathcal{M}\subset L_2$ and $L_1\cap\mathcal{M}$ is also dense in $L_p$. 
	
	Now assume that $\alpha_k\in\mathbb{C}$ for every $k\in\mathbb{N}_0$, and write $\alpha_k=\beta_k+i\gamma_k$ for some $\beta_k,\gamma_k\in\mathbb{R}$ for every $k$. Letting $\beta=(\beta_k)_{k=0}^{\infty}$ and $\gamma=(\gamma_k)_{k=0}^{\infty}$, it follows that $\beta,\gamma\subset\mathbb{R}$ (by definition) and $\beta,\gamma\in W_\infty$ since $$|\beta_k|,|\gamma_k|\leq\sqrt{|\beta_k|^2+|\gamma_k|^2}=|\alpha_k|\leq|\alpha|_{W_\infty}\text{ for every }k\in\mathbb{N}_0.$$ If we define $\alpha^*:=(\overline{\alpha_k})_{k=0}^{\infty}$, then $\beta=\frac{1}{2}(\alpha+\alpha^*)$ and $\gamma=\frac{1}{2i}(\alpha-\alpha^*)$.
	
	Fix $t\geq1$ and observe that
	$$\sup_{\lambda\in\mathbb{T}}\left|\frac{1}{t}\sum_{k=0}^{[t]-1}\alpha_k\lambda^k\right|
	=\sup_{\lambda\in\mathbb{T}}\left|\overline{\frac{1}{t}\sum_{k=0}^{[t]-1}\alpha_k\lambda^k}\right|
	=\sup_{\lambda\in\mathbb{T}}\left|\frac{1}{t}\sum_{k=0}^{[t]-1}\overline{\alpha_k}\overline{\lambda}^k\right|
	=\sup_{\lambda\in\mathbb{T}}\left|\frac{1}{t}\sum_{k=0}^{[t]-1}\overline{\alpha_k}\lambda^k\right|$$ since $\lambda\in\mathbb{T}$ if and only if $\overline{\lambda}\in\mathbb{T}$. From this one can find that
	\begin{align*}
	\sup_{\lambda\in\mathbb{T}}\left|\widehat{M}_{t}^{\beta}(\lambda)\right|
	=&\sup_{\lambda\in\mathbb{T}}\left|\frac{1}{2}\left(\widehat{M}_{t}^{\alpha}(\lambda)+\widehat{M}_{t}^{\alpha^*}(\lambda)\right)\right|\\ 
	\leq&\frac{1}{2}\left(\sup_{\lambda\in\mathbb{T}}\left|\widehat{M}_{t}^{\alpha}(\lambda)\right|+\sup_{\lambda\in\mathbb{T}}\left|\widehat{M}_{t}^{\alpha^*}(\lambda)\right|\right) \\
	=&\sup_{\lambda\in\mathbb{T}}\left|\widehat{M}_{t}^{\alpha}(\lambda)\right|
	\end{align*}
	and similarly 
	$\sup_{\lambda\in\mathbb{T}}\left|\widehat{M}_{t}^{\gamma}(\lambda)\right|\leq\sup_{\lambda\in\mathbb{T}}\left|\widehat{M}_{t}^{\alpha}(\lambda)\right|$.
	Hence if the polynomials $\widehat{M}_{t}^{\alpha}$ satisfy Inequality (\ref{ConvergenceCriteriaEquation}), then $\widehat{M}_{t}^{\beta}$ and $\widehat{M}_{t}^{\gamma}$ will as well. Therefore the claim for $\alpha$ follows by applying the previous case to $\beta$ and $\gamma$ since $\beta,\gamma\subset\mathbb{R}$ and both are in $W_\infty$ and then using the facts that
	$$M_n^{\alpha}(T)(x)=M_n^{\beta}(T)(x)+iM_n^{\gamma}(T)(x)$$ for every $n\in\mathbb{N}$ and that the sum of b.a.u. convergent sequences is also b.a.u. convergent to the sum of the b.a.u. limits.
	
	The proof and Lemma \ref{SubsequenceArgument} show that the b.a.u. limit must be $0$ if $x\in L_2$. To show that this is the case for each other $L_p$, we will first show that $\alpha$ must be a Hartman sequence. To see this, note first that the semifinite von Neumann algebra in the above arguments was arbitrary. Fix $\lambda\in\mathbb{T}$ and let $m$ denote the Lebesgue measure on $\mathbb{T}$. Consider the semifinite von Neumann algebra $(L_\infty(\mathbb{T},m),\int_{\mathbb{T}}\cdot d\mu)$ acting on $L_2(\mathbb{T},m)$ via multiplier operators and consider $T_{\lambda}\in DS^+(L_\infty(\mathbb{T},m),\int_{\mathbb{T}}\cdot dm)$ defined by
	$$T_{\lambda}f(\rho):=f(\lambda\rho)$$ for $f\in L_1(\mathbb{T},m)$ and $\rho\in\mathbb{T}$. Letting $\text{id}_{\mathbb{T}}(\rho)=\rho$ for $\rho\in\mathbb{T}$, it follows that $\text{id}_{\mathbb{T}}\in L_\infty(\mathbb{T},m)\subset L_2(\mathbb{T},m)$, meaning for $m$-a.e. $\rho\in\mathbb{T}$ we have by the above that
	$$0=\lim_{n\to\infty}\frac{1}{n}\sum_{k=0}^{n-1}\alpha_k (T_{\lambda}^k(\text{id}_{\mathbb{T}}))(\rho)=\lim_{n\to\infty}\frac{1}{n}\sum_{k=0}^{n-1}\alpha_k \text{id}_{\mathbb{T}}(\lambda^k\rho)=\rho\lim_{n\to\infty}\frac{1}{n}\sum_{k=0}^{n-1}\alpha_k\lambda^k,$$ implying $\lim_{n\to\infty}\frac{1}{n}\sum_{k=0}^{n-1}\alpha_k\lambda^k=0$ since the limit exists for some $\rho\neq0$. Since $\lambda\in\mathbb{T}$ was arbitrary it follows that $\alpha$ is a Hartman sequence.
	
	Fix $1<p<\infty$. Since $\alpha$ is a Hartman sequence, $L_p$ is reflexive, and $T$ is power bounded, it follows by \cite[Theorem 1.2]{lot} that the limit
	$$L(T,\alpha)(x):=\|\cdot\|_p-\lim_{n\to\infty}\frac{1}{n}\sum_{k=0}^{n-1}\alpha_kT^k(x)$$ exists for every $x\in L_p$, and so a well-known corollary to the Uniform Boundedness Principle implies that $L(T,\alpha)$ defines a continuous linear operator from $L_p$ to itself. Since $M_n^{\alpha}(T)(z)\to0$ b.a.u. as $n\to\infty$ for every $z\in L_1\cap\mathcal{M}$, and since b.a.u. limits and norm limits agree since both imply convergence in measure in $L_0$ (which is Hausdorff with this topology), it follows that $L(T,\alpha)(z)=0$ for all such $z$. Since $L_1\cap\mathcal{M}$ is dense in $L_p$, and since $L(T,\alpha)$ is continuous on $L_p$, it follows that $L(T,\alpha)(x)=0$ for every $x\in L_p(\mathcal{M},\tau)$ since it is the unique continuous linear extension of the map $L(T,\alpha):L_1\cap\mathcal{M}\to L_p$. Therefore the b.a.u. limit of $(M_n^{\alpha}(T)(x))_{n=1}^{\infty}$ must be $0$ as well since it must agree with the norm limit $L(T,\alpha)(x)=0$.  
	
	For the case $p=1$ we need to assume in addition that $\tau(\textbf{1})<\infty$. The proof for this case will follow similar reasoning as the previous case once convergence of the weighted averages with respect to $\|\cdot\|_1$ has been established.
	
	Fix $z\in\mathcal{M}$. As before one finds that the b.a.u. limit of $(M_n^{\alpha}(T)(z))_{n=1}^{\infty}$ must be $0$ since $\mathcal{M}$ is dense in $L_2$ when $\tau$ is finite. Since $\sup_{n\in\mathbb{N}}\|M_n^{\alpha}(T)(z)\|_\infty\leq|\alpha|_{W_\infty}\|z\|_\infty$, it follows by the noncommutative Bounded Convergence Theorem in \cite[Proposition 4.2]{comli3} that $\|M_n^{\alpha}(T)(z)\|_1\to0$ as $n\to\infty$. Since $\sup_{n\in\mathbb{N}}\|M_n^{\alpha}(T)\|_{L_1\to L_1}\leq|\alpha|_{W_\infty}<\infty$, it follows by \cite[Lemma 2.3]{comli1} that $$\mathcal{C}:=\{x\in L_1:(M_n^{\alpha}(T)(x))_{n=1}^{\infty}\text{ converges to with respect to }\|\cdot\|_1\}$$ is closed in $L_1$. Since we have shown that $\mathcal{M}\subset\mathcal{C}$, it follows that $\mathcal{C}=L_1$ since $\mathcal{C}$ is then both dense and closed in $L_1$, and this means that $(M_n^{\alpha}(T)(x))_{n=1}^{\infty}$ converges with respect to $\|\cdot\|_1$ for every $x\in L_1$. The rest of the proof follows as in the case where $1<p<\infty$.
\end{proof}

\begin{rem}\label{ReIndexing}
The polynomial $\widehat{M}_t^{\alpha}$ sometimes take other forms in the literature. For example, other sources may define it as either
$$\frac{1}{t}\sum_{k=1}^{[t]}\alpha_k\lambda^k \ \text{ or } \ \frac{1}{t}\sum_{k=0}^{[t]}\alpha_k\lambda^k\text{ for }t\geq1\text{ and }\lambda\in\mathbb{C}.$$ However, if $t\geq1$ and $\lambda\in\mathbb{T}$ then since
\begin{align*}
\left|\frac{1}{t}\sum_{k=0}^{[t]-1}\alpha_k\lambda^k\right|&\leq\left|\frac{1}{t}\sum_{k=1}^{[t]}\alpha_k\lambda^k\right|+\frac{|\alpha_0\lambda^{0}|+|\alpha_{[t]}\lambda^{[t]}|}{t}\leq\left|\frac{1}{t}\sum_{k=1}^{[t]}\alpha_k\lambda^k\right|+\frac{2|\alpha|_{W_\infty}}{t}\text{ and }\\
\left|\frac{1}{t}\sum_{k=0}^{[t]-1}\alpha_k\lambda^k\right|&\leq\left|\frac{1}{t}\sum_{k=0}^{[t]}\alpha_k\lambda^k\right|+\frac{|\alpha_{[t]}\lambda^{[t]}|}{t}\leq\left|\frac{1}{t}\sum_{k=0}^{[t]}\alpha_k\lambda^k\right|+\frac{|\alpha|_{W_\infty}}{t}
\end{align*}
one can easily see that if Inequality (\ref{ConvergenceCriteriaEquation}) holds for either one of these alternative definitions then the original assumption stated for our definition of $\widehat{M}_t^{\alpha}$ holds as well (H\"{o}lder's inequality on $\ell_2(\mathbb{N}_0)$ may be needed when squaring the inequalities).

Such sources may change the index set and work with sequences of the form $(\beta_j)_{j=1}^{\infty}$ instead of $(\alpha_k)_{k=0}^{\infty}$ (like in the first case above). This also will not cause issues in the above since if we use $\alpha_k=\beta_{k+1}$ for every $k\in\mathbb{N}_0$ then for every $\lambda\in\mathbb{T}$ it follows that
$$\left|\frac{1}{t}\sum_{k=0}^{[t]-1}\alpha_k\lambda^k\right|
=|\lambda|\left|\frac{1}{t}\sum_{k=0}^{[t]-1}\alpha_k\lambda^k\right|
=\left|\frac{1}{t}\sum_{k=0}^{[t]-1}\alpha_k\lambda^{k+1}\right|
=\left|\frac{1}{t}\sum_{j=1}^{[t]}\beta_j\lambda^j\right|.
$$
\end{rem}

\begin{rem}
If a variant of Lemma \ref{SubsequenceArgument} could be proven for a.u. convergence, then a.u. convergence would hold in Theorem \ref{ConvergenceCriteriaI} whenever $p\geq 2$. However, since the fact that $0\leq a\leq b$ in $L_0(\mathcal{M},\tau)^+$ implies $0\leq eae\leq ebe$ and $0\leq e(b-a)e\leq\|e(b-a)e\|_\infty\textbf{1}$ in $\mathcal{M}^{+}$ for appropriately chosen $e\in\mathcal{P(M)}$ plays an important role in the proof of Lemma \ref{SubsequenceArgument}, a very different approach would likely be needed.
\end{rem}

\subsection{Applications}\label{Applications}
A number of weighted ergodic theorems in classical ergodic theory are proven by showing that the Inequality (\ref{ConvergenceCriteriaEquation}) holds for the weight sequence $\alpha$. From the list in the introduction, some sequences that are now good weights in the von Neumann algebra setting include:
\begin{itemize}
\item[(iii)] special cases of quadratic polynomial powers of unit circle elements \cite[Proposition 8]{assani1},
\item[(v)] special cases of dynamically generated sequences \cite[Theorem 4]{assani1},
\item[(vi)] ones generated by a bounded i.i.d. sequence \cite[Theorem 2]{assani1},
\item[(vii)] one generated by the M\"{o}bius function $\mu$ \cite[Equation (1)]{davenport}, \cite[Proposition 3.1]{eakpldlr}, \cite[Lemma 1]{BatemanChowla}, 
\item[(viii)] one generated by the Lioville function $\lambda$ \cite[Lemma 1]{BatemanChowla}, 
\item[(ix)] a totally-balanced $k$-automatic sequence \cite[Proposition 8.1]{ek} or an invertible $k$-automatic sequence \cite[Propositions 9.2 and 9.3]{ek}, 
\item[(x)] certain types of random $q$-multiplicative sequences \cite[Theorems 5 and 6]{fan}. 
\end{itemize}
The class of (unbounded) Besicovitch sequences considered in \cite[Theorem 1.12]{demeterjones} which does not satisfying the usual duality assumptions can also be shown to be good weights by following similar reasoning.

We will now show that i.i.d. sequences are good weights for the individual ergodic theorem on von Neumann algebras similar to the sense mentioned in the introduction. We will use use the results of \cite{assani1} to expedite the process so that we can emphasize some technical points about a few proofs of the argument for them that are not usually explicitly mentioned in the literature that we feel should be acknowledged somewhere (see Remark \ref{WhyOnlySubsequenceAndNotAllReasoning}). This also gives another good class of unbounded weights that can be used in weighted individual ergodic theorems in the von Neumann algebra setting (after the ones considered in \cite{ob2,ob4}), extending case (xiii) in addition to those mentioned above; our proof mirrors the commutative version's proof in this case as well. 

Afterwards we outline the general process for the other sequences. This general method has a counterpart in the commutative setting that was seen in \cite[Proposition 3.1]{eakpldlr} (where it was stated and proved for the M\"{o}bius function) and \cite[Corollaries 1 and 2]{fan} (for more general weights).

For the rest of this section we will let $\mathcal{B}_{\mathbb{C}}$ denote the Borel $\sigma$-algebra of $\mathbb{C}$ when it is equipped with its standard metric topology.

\begin{df}
	Let $(\Omega,\mathcal{F},\nu)$ be a probability space. Let $Y_k:\Omega\to\mathbb{C}$ be a measurable function for every $k\in\mathbb{N}_0$. Then the sequence $Y=(Y_k)_{k=0}^{\infty}$ is called an \textit{independent identically distributed sequence} (abbreviated as i.i.d. sequence) on $(\Omega,\mathcal{F},\mu)$ if it satisfies:
	\begin{itemize}
		\item[(i)] (Independent) $\nu\left(\bigcap_{j=1}^{m}Y_{k_j}^{-1}(E_{j})\right)=\prod_{j=1}^{m}\nu(Y_{k_j}^{-1}(E_{j}))$ for every $E_1,...,E_m\in\mathcal{B}_{\mathbb{C}}$, every $k_1,...,k_m\in\mathbb{N}_0$, and every $m\in\mathbb{N}$.
		\item[(ii)] (Identically Distributed) $\nu\circ Y_k^{-1}=\nu\circ Y_j^{-1}$ as measures on $(\mathbb{C},\mathcal{B}_{\mathbb{C}})$ for every $k,j\in\mathbb{N}_0$.
	\end{itemize}
	For every $\omega\in\Omega$ we will write $Y(\omega):=(Y_k(\omega))_{k=0}^{\infty}\subset\mathbb{C}$. 
\end{df}
In probability theory it is common to write $\mathbb{E}(Y_k)$ for the integral $\int_{\Omega}Y_k(\omega)d\nu(\omega)$. For i.i.d. sequences it is known that $\mathbb{E}(f(Y_k))=\mathbb{E}(f(Y_0))$ for every $k\in\mathbb{N}_0$ and every Borel-measurable function $f:\mathbb{C}\to\mathbb{C}$ whenever the expectation is well-defined. In particular, this holds if $f(t)=t+c$ for some fixed $c\in\mathbb{C}$, $f(t)=|t|^q$ for some fixed $q\in[1,\infty)$, and $f(t)=\chi_{E}(t)t$ for some fixed $E\in\mathcal{F}$.

If $\|Y_0\|_\infty<\infty$ then the fact that $Y$ is identically distributed implies $Y(\omega)\in W_\infty$ and $\|Y(\omega)\|_{W_\infty}\leq\|Y_0\|_\infty$ for $\nu$-a.e. $\omega\in\Omega$. Similarly, if $\mathbb{E}(|Y_0|^q)<\infty$ for some $1\leq q<\infty$ then the i.i.d. assumption and the strong law of large numbers implies that $Y(\omega)\in W_q$ and $\|Y(\omega)\|_{W_q}=\mathbb{E}(|Y_0|^1)^{1/q}$ for $\nu$-a.e. $\omega\in\Omega$.

If $f:\mathbb{C}\to\mathbb{C}$ is a Borel measurable function and $Y=(Y_k)_{k=0}^{\infty}$ is an i.i.d. sequence on $(\Omega,\mathcal{F},\nu)$, then one finds that $(f\circ Y_k)_{k=0}^{\infty}$ is an i.i.d. sequence on $(\Omega,\mathcal{F},\nu)$ as well; we will denote this sequence by $f\circ Y$. Importantly, if $f$ is bounded then $\|f\circ Y_0\|_\infty<\infty$ even if $Y_0$ is not bounded.

\begin{teo}\label{BoundedIID}
Let $(\Omega,\mathcal{F},\nu)$ be a probability space and $Y=(Y_k)_{k=0}^{\infty}$ be a sequence of bounded i.i.d. random variables on $(\Omega,\mathcal{F},\nu)$. Then there exists a set $\Omega^{\prime}\in\mathcal{F}$ with $\nu(\Omega^{\prime})=1$ such that for every $\omega\in\Omega^{\prime}$: if $(\mathcal{M},\tau)$ is a semifinite von Neumann algebra, $T\in DS^+(\mathcal{M},\tau)$, and $1\leq p<\infty$, then the weighted averages
$$M_n^{Y(\omega)}(T)(x):=\frac{1}{n}\sum_{k=0}^{n-1}Y_k(\omega)T^k(x)$$ converge b.a.u. as $n\to\infty$ for every $x\in L_p(\mathcal{M},\tau)$. 

Furthermore, if $x\in L_p(\mathcal{M},\tau)$ and if $F_T(x)\in L_p(\mathcal{M},\tau)$ denotes the b.a.u. limit of the unweighted averages $(M_n(T)(x))_{n=1}^{\infty}$, then the b.a.u. limit of $(M_n^{Y(\omega)}(T)(x))_{n=1}^{\infty}$ is equal to $\mathbb{E}(Y_0)F_T(x)$ if $p>1$ or if $p=1$ and $\tau$ is finite.

\end{teo}
\begin{proof}
Since $Re,Im:\mathbb{C}\to\mathbb{R}$ are both continuous functions, and since $Y_k=Re(Y_k)+iIm(Y_k)$ for every $k\in\mathbb{N}_0$, it suffices to prove the claim for $Re\circ Y$ and $Im\circ Y$ (both of which will be bounded i.i.d. sequences on $(\Omega,\mathcal{F},\nu)$ since $Y$ is), as the claim for general $Y$ will then follow from the linearity of b.a.u. limits. As such, we may assume that each $Y_k$ is real-valued.

Let $f:\mathbb{R}\to\mathbb{R}$ be given by $f(y)=y-\mathbb{E}(Y_0)$, noting that $f$ is continuous, and write $Z_k=f\circ Y_k$ and $Z=(Z_k)_{k=0}^{\infty}$. Since $Y$ is an i.i.d. sequence on $(\Omega,\mathcal{F},\nu)$, it follows by the above that $f\circ Y=Z$ is as well. Furthermore, since $\mathbb{E}(Y_k)=\mathbb{E}(Y_0)$ for all $k\in\mathbb{N}_0$ from $Y$ being an i.i.d. sequence, it follows that $\mathbb{E}(Z_k)=\mathbb{E}(Y_k-\mathbb{E}(Y_0))=\mathbb{E}(Y_k)-\mathbb{E}(Y_0)=0$ for all $k\in\mathbb{N}_0$. Since $|Y_0(\omega)|\leq\|Y_0\|_\infty$, it follows that $|Z_0(\omega)|=|(f\circ Y_0)(\omega)|=|Y_0(\omega)-\mathbb{E}(Y_0)|\leq2\|Y_0\|_\infty$ $\mu$-a.e.; hence $Z=(Z_k)_{k=0}^{\infty}$ is a bounded i.i.d. on $(\Omega,\mathcal{F},\mu)$ with $\mathbb{E}(Z_0)=0$. Furthermore, being a bounded i.i.d. sequence on a probability space implies that $\mathbb{E}(|Z_0|^2)\leq4\|Y_0\|_\infty^2$, 
so that $Z$ has bounded second moment.

Since for every $n\in\mathbb{N}$ and $\omega\in\Omega$ the function $P_n:\mathbb{T}\to\mathbb{C}$ given by $$P_n(\lambda):=\frac{1}{n}\sum_{k=0}^{n-1}(Y_k(\omega)-\mathbb{E}(Y_0))\lambda^k$$ is continuous (in the variable $\lambda$), it follows that the function $f_n:\Omega\to\mathbb{C}$ given by $$f_n(\omega):=\sup_{\lambda\in\mathbb{T}}\left|\frac{1}{n}\sum_{k=0}^{n-1}(Y_k(\omega)-\mathbb{E}(Y_0))\lambda^k\right|^2$$ is measurable (in the variable $\omega$) and defines an element of $L_\infty(\Omega,\mathcal{F},\nu)$ since $Y$ is bounded. After some small rewriting of the indices in \cite[Theorem 1]{assani1} (as done in Remark \ref{ReIndexing}) these conditions on $Z$ imply that
$$\int_{\Omega}\sup_{\lambda\in\mathbb{T}}\left|\frac{1}{n}\sum_{k=0}^{n-1}Z_k(\omega)\lambda^k\right|^2d\nu(\omega) 
\leq\frac{(1+\sqrt{2})\mathbb{E}(|Z_0|^2)}{\sqrt{n}}$$
for every $n\in\mathbb{N}$. Therefore if $\rho>1$ then it follows that
\begin{align*}
\int_{\Omega}\sum_{m=1}^{\infty}\sup_{\lambda\in\mathbb{T}}\Bigg|\frac{1}{\rho^m}\sum_{k=0}^{[\rho^m]-1}(Y_k(\omega)&-\mathbb{E}(Y_0))\lambda^k\Bigg|^2d\nu(\omega) \\
\leq&\sum_{m=1}^{\infty}\int_{\Omega}\sup_{\lambda\in\mathbb{T}}\left|\frac{1}{[\rho^m]}\sum_{k=0}^{[\rho^m]-1}(Y_k(\omega)-\mathbb{E}(Y_0))\lambda^k\right|^2d\nu(\omega) \\
=&\sum_{m=1}^{\infty}\int_{\Omega}\sup_{\lambda\in\mathbb{T}}\left|\frac{1}{[\rho^m]}\sum_{k=0}^{[\rho^m]-1}Z_k(\omega)\lambda^k\right|^2d\nu(\omega) \\
\leq&\sum_{m=1}^{\infty}\frac{(1+\sqrt{2})\mathbb{E}(|Z_0|^2)}{\sqrt{[\rho^m]}} \\
\leq&\sum_{m=1}^{\infty}\frac{4(1+\sqrt{2})\|Y_0\|_\infty^2}{\sqrt{[\rho^m]}} \\
\leq&4\sqrt{2}(1+\sqrt{2})\|Y_0\|_\infty^2\sum_{m=1}^{\infty}\left(\frac{1}{\sqrt{\rho}}\right)^m<\infty,
\end{align*}
where the first and last inequalities follow from the fact that $1\leq [\rho^m]\leq\rho^m\leq2[\rho^m]$ implies that $\frac{1}{\rho^m}\leq\frac{1}{[\rho^m]}$ and $\frac{1}{\sqrt{[\rho^m]}}\leq\sqrt{\frac{2}{\rho^m}}=\sqrt{2}\left(\frac{1}{\sqrt{\rho}}\right)^m$ for every $m\in\mathbb{N}$, and then using the fact that $\frac{1}{\sqrt{\rho}}<1$ implies the geometric series converges. Hence, for every $\rho>1$ there exists a set $\Omega_{\rho}\in\mathcal{F}$ with $\nu(\Omega_{\rho})=1$ satisfying $$\sum_{m=1}^{\infty}\sup_{\lambda\in\mathbb{T}}\left|\frac{1}{\rho^m}\sum_{k=0}^{[\rho^m]-1}\Big(Y_k(\omega)-\mathbb{E}(Y_0)\Big)\lambda^k\right|^2<\infty$$ for every $\omega\in\Omega_{\rho}$. 

For every $j\in\mathbb{N}_0$ define $\rho_j:=2^{1/2^j}$, and let $\Omega^{\prime}:=\bigcap_{j=0}^{\infty}\Omega_{\rho_j}$, noting that $\nu(\Omega^{\prime})=1$. Then it follows that $(\rho_j)_{j=0}^{\infty}\subset(1,\infty)$ satisfies the conditions of Theorem \ref{ConvergenceCriteriaI} for the sequence $Z(\omega)\in W_\infty$ whenever $\omega\in\Omega^{\prime}$, implying $M_n^{Z(\omega)}(T)(x)$ converges b.a.u. as $n\to\infty$ for every semifinite von Neumann algebra $(\mathcal{M},\tau)$, $T\in DS^+(\mathcal{M},\tau)$, $1\leq p<\infty$, and $x\in L_p(\mathcal{M},\tau)$.

Keeping the notations from the previous paragraph, assume further that either $1<p<\infty$ or $p=1$ and $\tau$ is finite. Then the b.a.u. limit of $(M_n^{Z(\omega)}(T)(x))_{n=1}^{\infty}$ will be $0$ for each $\omega\in\Omega^{\prime}$; therefore, since $$M_n^{Y(\omega)}(T)(x)=M_n^{Z(\omega)}(T)(x)+\mathbb{E}(Y_0)M_n(T)(x)$$ for every $n\in\mathbb{N}$, and since the latter two sequences converge b.a.u. to $0$ and $\mathbb{E}(Y_0)F_T(x)$ respectively as $n\to\infty$, it follows that $(M_n^{Y(\omega)}(T)(x))_{n=1}^{\infty}$ converges to $\mathbb{E}(Y_0)F_T(x)$ b.a.u. as $n\to\infty$ for all $x\in L_p(\mathcal{M},\tau)$ with $p$ and $\tau$ as above.
\end{proof}

\begin{rem}\label{WhyOnlySubsequenceAndNotAllReasoning}
We wish to note that the paragraph defining the $\rho_j$'s and $\Omega^{\prime}$ is where our phrasing of Lemma \ref{SubsequenceArgument} was needed. In particular, this particular wording was needed to guarantee that the set $\Omega^{\prime}$ was actually measurable since one needs to take the intersection of sets to find it. While it may be the case that $\bigcap_{\rho>1}\Omega_{\rho}$ is measurable with full measure, the argument above doesn't immediately prove that it is, so using countably many sets is a much easier way to guarantee that the set $\Omega^{\prime}$ found in the theorem is measurable with $\mu(\Omega^{\prime})=1$.
\end{rem}

We will now show that the assumption of the i.i.d. sequence under consideration being bounded can be replaced by it having finite $q$-th moment, extending case (xiii) from the list in the introduction to the noncommutative setting.

\begin{teo}\label{UnboundedIID-Duality}
Let $1<p,q<\infty$ be such that $\frac{1}{p}+\frac{1}{q}=1$. Let $(\Omega,\mathcal{F},\nu)$ be a probability space and let $Y=(Y_k)_{k=0}^{\infty}$ be an i.i.d. sequence of random variables on $(\Omega,\mathcal{F},\nu)$ such that  $\mathbb{E}(|Y_0|^q)<\infty$.

Then there exists a set $\Omega^{\prime}\in\mathcal{F}$ with $\mu(\Omega^{\prime})=1$ such that for every $\omega\in\Omega^{\prime}$: if $(\mathcal{M},\tau)$ is a semifinite von Neumann algebra and $T\in DS^+(\mathcal{M},\tau)$, then the weighted ergodic averages
$$M_n^{Y(\omega)}(T)(x):=\frac{1}{n}\sum_{k=0}^{n-1}Y_k(\omega)T^k(x)$$ converge b.a.u. as $n\to\infty$ for every $x\in L_p(\mathcal{M},\tau)$. Furthermore, if $F_T(x)\in L_p(\mathcal{M},\tau)$ denotes the b.a.u. limit of $(M_n(T)(x))_{n=1}^{\infty}$, then the b.a.u. limit of $(M_n^{Y(\omega)}(T)(x))_{n=1}^{\infty}$ is equal to $\mathbb{E}(Y_0)F_T(x)$ for every $\omega\in\Omega^{\prime}$.

\end{teo}
\begin{proof}
Since $(Y_k)_{k=0}^{\infty}$ being an i.i.d. sequence on $(\Omega,\mathcal{F},\nu)$ implies that $(f(Y_k))_{k=0}^{\infty}$ is an i.i.d. sequence on $(\Omega,\mathcal{F},\nu)$ for every measurable function $f:\mathbb{C}\to\mathbb{C}$ as well, we may assume that $Y_k$ is nonnegative for every $k\in\mathbb{N}_0$ by proving the claim for $A_{k,1}:=Re(Y_k)^{+}$, $A_{k,2}:=Im(Y_k)^{+}$, $A_{k,3}:=Re(Y_k)^{-}$, and $A_{k,4}:=Im(Y_k)^{-}$ (where $Re(Y_k)^{+}:=\chi_{[0,\infty)}(Re(Y_k))Re(Y_k)$ and the others being defined similarly) and then using the fact that $Y_k=\sum_{j=0}^{3}i^j A_{k,j}$ as functions on $\Omega$ for every $k\in\mathbb{N}_0$ and that $\mathbb{E}(Y_0)=\sum_{j=0}^{3}i^j\mathbb{E}(A_{0,j})$.

Fix $t>0$ and define $Z_k^{(t)}(\omega):=\chi_{[0,t]}(Y_k(\omega))Y_k(\omega)$ for every $k\in\mathbb{N}_0$; note that $Z_k^{(t)}=g_t(Y_k)$ for the measurable function $g_t:[0,\infty)\to[0,\infty)$ given by $g_t(y)=\chi_{[0,t]}(y)y$ for $y\in[0,\infty)$, so that $Z^{(t)}:=(Z_k^{(t)})_{k=0}^{\infty}$ is a bounded i.i.d. sequence on $(\Omega,\mathcal{F},\mu)$ for every $t>0$.

By using the (measurable) function $f_t:[0,\infty)\to\mathbb{R}$ defined by $f_t(y)=y-\chi_{[0,t]}(y)y=\chi_{(t,\infty)}(y)y$ on $[0,\infty)$ one finds that $(|Y_k-Z_k^{(t)}|)_{k=0}^{\infty}=(f_t(Y_k))_{k=0}^{\infty}=f_t\circ Y$ is an i.i.d. sequence on $(\Omega,\mathcal{F},\nu)$. Hence by the strong law of large numbers, H\"{o}lder's inequality, and Chebyshev's Inequality (which states that for every $t>0$ it follows that $\nu(\{\omega\in\Omega:|Y_0(\omega)|\geq t\}) \leq\frac{\|Y_0\|_1}{t}$), there exists a set $\Omega_t\in\mathcal{F}$ with $\nu(\Omega_t)=1$ such that for every $\omega\in\Omega_t$ it follows that
\begin{align*}
	\lim_{n\to\infty}\frac{1}{n}\sum_{k=0}^{n-1}|Y_k&(\omega)-Z_k^{(t)}(\omega)|
	=\mathbb{E}(|Y_0-Z_0^{(t)}|) \\
	=&\int_{\Omega}|Y_0(\omega)-Z_0^{(t)}(\omega)|d\nu(\omega) \\
	=&\int_{|Y_0|^{-1}((t,\infty))}|Y_0(\omega)|d\nu(\omega) \\
	\leq&\left(\int_{|Y_0|^{-1}((t,\infty))}|Y_0(\omega)|^qd\nu(\omega)\right)^{1/q}\left(\int_{|Y_0|^{-1}((t,\infty))}1^{p}d\nu(\omega)\right)^{1/p} 
\end{align*}
\begin{align*}
	\leq&\|Y_0\|_q\nu(\{\omega\in\Omega:|Y_0(\omega)|>t\})^{1/p} \\
	\leq&\frac{\|Y_0\|_q\|Y_0\|_1^{1/p}}{t^{1/p}}\leq\frac{\|Y_0\|_q^{1+\frac{1}{p}}}{t^{1/p}}
\end{align*}
(where the last inequality follows from the fact that $(\Omega,\mathcal{F},\nu)$ being a probability space implies $\|Y_0\|_1\leq\|Y_0\|_q<\infty$). Therefore
$$\|Y(\omega)-Z^{(t)}(\omega)\|_{W_1}=\limsup_{n\to\infty}\frac{1}{n}\sum_{k=0}^{n-1}|Y_k(\omega)-Z_k^{(t)}(\omega)|\leq\frac{\|Y_0\|_q^{1+\frac{1}{p}}}{t^{1/p}}$$ for every $\omega\in\Omega_{t}$ and every $t>0$.

For every $r\in\mathbb{N}_0$ let $\Omega_{r}^{\prime}\in\mathcal{F}$ be the set with $\nu(\Omega_{r}^{\prime})=1$ obtained in Theorem \ref{BoundedIID} associated to the bounded i.i.d. sequence $(Z_k^{(r)})_{k=0}^{\infty}$ on $(\Omega,\mathcal{F},\nu)$. 
If we define $\Omega^{\prime}:=\bigcap_{r=1}^{\infty}(\Omega_{r}\cap\Omega_{r}^{\prime})$, then $\Omega^{\prime}\in\mathcal{F}$ and $\nu(\Omega^{\prime})=1$.

Fix $\omega\in\Omega^{\prime}$, $x\in L_1\cap\mathcal{M}$, and $\epsilon>0$. Let $F_T(x)\in L_0$ be such that $M_n(T)(x)\to F_T(x)$ b.a.u. as $n\to\infty$. Since $M_n^{Z^{(t)}(\omega)}(T)(x)\to\mathbb{E}(Z_0^{(t)})F_T(x)$ and $\mathbb{E}(Z_0^{(t)})M_n(T)(x)\to\mathbb{E}(Z_0^{(t)})F_T(x)$ b.a.u. as $n\to\infty$ for every $t>0$, we find that for every $r\in\mathbb{N}$ there exists $e_r\in\mathcal{P(M)}$ such that $\tau(e_r^{\perp})\leq\frac{\epsilon}{2^{r+1}}$ and $\|e_r(M_n^{Z^{(r)}(\omega)}(T)(x)-\mathbb{E}(Z_0^{(r)})M_n(T)(x))e_r\|_\infty\to0$ as $n\to\infty$. Letting $e=\bigwedge_{r=1}^{\infty}e_r$, we have that $\tau(e^\perp)\leq\epsilon$ and $$\left\|e\left(M_n^{Z^{(r)}(\omega)}(T)(x)-\mathbb{E}(Z_0^{(r)})M_n(T)(x)\right)e\right\|_\infty\to0\text{ as }n\to\infty$$ for every $r\in\mathbb{N}_0$. Note that $x\in L_1\cap\mathcal{M}$ implies that $\|eT^k(x)e\|_\infty\leq\|x\|_\infty$ and $\|eM_n(T)(x)e\|_\infty\leq\|x\|_\infty$ for all $k\in\mathbb{N}_0$ and $n\in\mathbb{N}$.

Fix $\delta>0$.

Since $0\leq Z_0^{(s)}(\omega)\leq Z_0^{(t)}(\omega)$ for $0<s\leq t$ and $Z_0^{(t)}(\omega)\to Y_0(\omega)$ as $t\to\infty$ for $\nu$-a.e. $\omega\in\Omega$, it follows by Lebesgue's Monotone Convergence Theorem that $\lim_{r\to\infty}\mathbb{E}(Z_0^{(r)})=\mathbb{E}(Y_0)$, so that there exists $R_1\in\mathbb{N}_0$ such that $r\geq R_1$ implies $$|\mathbb{E}(Y_0)-\mathbb{E}(Z_0^{(r)})|<\frac{\delta}{3(\|x\|_\infty+1)}.$$

Let $R_2\in\mathbb{N}_0$ be such that $r\geq R_2$ implies
$$\frac{1}{r^{1/p}}\leq\frac{\delta}{6\left(\|Y_0\|_q^{1+\frac{1}{p}}+1\right)(\|x\|_\infty+1)}.$$ 

Let $R=\max\{R_1,R_2\}$. Then since $\omega\in\Omega_{R}$ it follows that there exists $N_1\in\mathbb{N}_0$ such that $n\geq N_1$ implies $$\frac{1}{n}\sum_{k=0}^{n-1}|Y_k(\omega)-Z_0^{(R)}(\omega)|\leq\frac{\|Y_0\|_q^{1+\frac{1}{p}}}{R^{1/p}}+\frac{\delta}{6(\|x\|_\infty+1)}\leq\frac{\delta}{3(\|x\|_\infty+1)}.$$

Since $\omega\in\Omega_{R}^{\prime}$, there exists $N_2\in\mathbb{N}_0$ such that $n\geq N_2$ implies
$$\|e(M_n^{Z^{(R)}(\omega)}(T)(x)-\mathbb{E}(Z_0^{(R)})M_n(T)(x))e\|_\infty<\frac{\delta}{3}.$$ If $N=\max\{N_1,N_2\}$, then $n\geq N$ implies
\begin{align*}
	\|e(M_n^{Y(\omega)}(T)(x)-\mathbb{E}(Y_0)M_n(T)&(x))e\|_\infty \\
	\leq&\|e(M_n^{Y(\omega)}(T)(x)-M_n^{Z^{(R)}(\omega)}(T)(x))e\|_\infty \\
	&+\|e(M_n^{Z^{(R)}(\omega)}(T)(x)-\mathbb{E}(Z_0^{(R)})M_n(T)(x))e\|_\infty \\
	&+\|e(\mathbb{E}(Z_0^{(R)})M_n(T)(x)-\mathbb{E}(Y_0)M_n(T)(x))e\|_\infty \\
	\leq&\frac{1}{n}\sum_{k=0}^{n-1}|Y_k(\omega)-Z_k^{(R)}(\omega)|\|eT^k(x)e\|_\infty \\
	&+\|e(M_n^{Z^{(R)}(\omega)}(T)(x)-\mathbb{E}(Z_0^{(R)})M_n(T)(x))e\|_\infty \\
	&+|\mathbb{E}(Z_0^{(R)})-\mathbb{E}(Y_0)|\|x\|_\infty \\
	\leq&\frac{\delta\|x\|_\infty}{3(\|x\|_\infty+1)}+\frac{\delta}{3}+\frac{\delta\|x\|_\infty}{3(\|x\|_\infty+1)}\leq\delta.
\end{align*}

Since $n\geq N$ and $\delta>0$ were arbitrary, it follows that $$\left\|e\left(M_n^{Y(\omega)}(T)(x)-\mathbb{E}(Y_0)M_n(T)(x)\right)e\right\|_\infty\to0$$ as $n\to\infty$, and since $\epsilon>0$ was arbitrary it follows that $$M_n^{Y(\omega)}(T)(x)-\mathbb{E}(Y_0)M_n(T)(x)\to0\text{ b.a.u. as }n\to\infty.$$ Since $\mathbb{E}(Y_0)M_n(T)(x)\to\mathbb{E}(Y_0)F_T(x)$ b.a.u. as $n\to\infty$ by \cite[Theorem 6.3]{jx}, it follows that $M_n^{Y(\omega)}(T)(x)\to\mathbb{E}(Y_0)F_T(x)$ b.a.u. as $n\to\infty$ as well.

Since $x\in L_1\cap\mathcal{M}$ was arbitrary, it follows that $(M_n^{Y(\omega)}(T)(x))_{n=1}^{\infty}$ converges b.a.u. as $n\to\infty$ for all $x\in L_p(\mathcal{M},\tau)$ and $\omega\in\Omega^{\prime}$ by \cite[Corollary 3.5]{ob4} and \cite[Theorem 3.4]{ob2} since $Y(\omega)\in W_q$. For similar reasons as in Theorem \ref{ConvergenceCriteriaI} it follows that the b.a.u. limit will be $\mathbb{E}(Y_0)F_T(x)$ for every $x\in L_p$ since $\mathbb{E}(Y_0)F_T=L(T,Y(\omega))$ is a bounded linear operator from $L_p(\mathcal{M},\tau)$ to itself for every $\omega\in\Omega^{\prime}$.
\end{proof}

As was mentioned previously, the same methods apply to a number of other sequences, a few of which we gather in the following result.

\begin{teo}
Let $(\mathcal{M},\tau)$ be a semifinite von Neumann algebra, $1\leq p<\infty$, and $T\in DS^+(\mathcal{M},\tau)$. If $\alpha=(\alpha_k)_{k=0}^{\infty}\in W_\infty$ is defined by one of the following:
\begin{itemize}
	\item[(vii)] $\alpha_k=\mu(k)$, where $\mu:\mathbb{N}_0\to\{-1,0,1\}$ denotes the M\"{o}bius function,
	\item[(viii)] $\alpha_k=\lambda(k)$, where $\lambda:\mathbb{N}_0\to\{-1,1\}$ denotes the Liouville function,
	\item[(ix)] $\alpha$ is either a totally-balanced $k$-automatic sequence or an invertible $k$-automatic sequence as defined in \cite{ek},
\end{itemize}
then the averages
$$\frac{1}{n}\sum_{k=0}^{n-1}\alpha_kT^k(x)\to0\text{ b.a.u. as }n\to\infty$$ for every $x\in L_p(\mathcal{M},\tau)$.
\end{teo}
\begin{proof}
In each case of $\alpha$ above for every $A>0$ there exists $C_A>0$ such that
$$\sup_{\lambda\in\mathbb{T}}\left|\frac{1}{x}\sum_{k=0}^{[x]-1}\alpha_k\lambda^k\right|\leq\frac{C_A}{(\ln(x))^{A}}\text{ for every }x>0.$$ 
Cases (vii) and (viii) were shown in \cite{davenport}, while case (ix) was shown in \cite[Propositions 8.1, 9.2, and 9.3]{ek} (in \cite{ek} it actually shown that for every totally-balanced $k$-automatic sequence $\alpha$ there exists $C>0$ and $K_C>0$ such that the upper bound above is $K_C/x^C$ any every $x>0$, but this implies the above anyways). Hence if $\rho>1$ and $A=2$ then
$$\sum_{m=1}^{\infty}\sup_{\lambda\in\mathbb{T}}\left|\frac{1}{\rho^m}\sum_{k=0}^{[\rho^m]-1}\alpha_k\lambda^k\right|^2
\leq\sum_{m=1}^{\infty}\left(\frac{C_2}{(\ln(\rho^m))^{2}}\right)^{2}
=\frac{C_2^2}{\ln(\rho)^{4}}\sum_{m=1}^{\infty}\frac{1}{m^{4}}<\infty.
$$ 
Since each case of sequence is bounded, and since $\rho>1$ was arbitrary, the result follows from Theorem \ref{ConvergenceCriteriaI} by using $\rho_j:=2^{1/2^j}$ for every $j\in\mathbb{N}_0$.
\end{proof}

\section{$T$-admissible Processes}\label{TAdm}

Throughout this section, unless otherwise stated we will assume that $(\mathcal{M},\tau)$ is a semifinite von Neumann algebra and $1\leq p<\infty$. In this section we will focus our attention on the case where $T:\mathcal{M}\to\mathcal{M}$ is a normal $\tau$-preserving $*$-automorphism (and will use $T$ for its isometric extensions to $L_r(\mathcal{M},\tau)$ for every $1\leq r\leq\infty$).

Following the language of \cite{comli2}, a sequence $x=(x_k)_{k=0}^{\infty}\subset L_p(\mathcal{M},\tau)^h$ is called a \textit{$T$-admissible process} if $T(x_k)\leq x_{k+1}$ for every $k\in\mathbb{N}_0$, and it is \textit{strongly $p$-bounded} if $\sup_{k\in\mathbb{N}_0}\|x_k\|_p<\infty$.

If $1\leq p,q\leq\infty$, $x=(x_k)_{k=0}^{\infty}\subset L_p(\mathcal{M},\tau)$, and $\alpha=(\alpha_k)_{k=0}^{\infty}\in W_q$, then write
$$A_n^{\alpha}(x):=\frac{1}{n}\sum_{k=0}^{n-1}\alpha_k x_k$$ for the $n$-th average of the sequence $x$ weighted by $\alpha$. If $y\in L_p(\mathcal{M},\tau)$ and $T\in DS^+(\mathcal{M},\tau)$ are fixed and if $x$ above is given by $x_k=T^k(y)$ for every $k\in\mathbb{N}_0$, then $A_n^{\alpha}(x)=M_n^{\alpha}(T)(y)$ for every $n\in\mathbb{N}$.

\begin{lm}\label{l31}
	Assume $x=(x_k)_{k=0}^{\infty}\subset L_p(\mathcal{M},\tau)^+$ is a strongly $p$-bounded $T$-admissible process. Then there exists $w\in L_p(\mathcal{M},\tau)^+$ such that $\|w-T^{-k}(x_k)\|_p\to0$ as $k\to\infty$ and $x_k\leq T^k(w)$ for every $k\in\mathbb{N}_0$.
\end{lm}
\begin{proof}
	Since $x$ is strongly $p$-bounded and $T$-admissible, and since $T^{-k}$ is positive for every $k\geq0$, it follows that $(T^{-k}(x_k))_{k=0}^{\infty}$ is strongly $p$-bounded and increasing (with respect to the ordering on $L_0^+$). It was shown in \cite[Proposition 3.2 and Theorem 5.15]{ddp} that $L_p(\mathcal{M},\tau)$ has the Beppo-Levi property when $1\leq p<\infty$; together with what we've shown thus far, this implies that $w:=\sup_{k\in\mathbb{N}_0}T^{-k}(x_k)$ exists and is in $L_p^+$, that $\|w\|_p=\sup_{k\in\mathbb{N}_0}\|T^{-k}(x_k)\|_p=\sup_{k\in\mathbb{N}_0}\|x_k\|_p$, and that $\|w-T^{-k}(x_k)\|_p\to0$ as $k\to\infty$. The supremum property of $w$ then also implies that $x_k\leq T^k(w)$ for every $k\in\mathbb{N}_0$ since $T^k$ is positive.
\end{proof}

\begin{df} The operator $w\in L_p^{+}$ obtained in Lemma \ref{l31} will be called the \textit{exact dominant} of $x=(x_k)_{k=0}^{\infty}$. \end{df}

For any $1\leq q\leq\infty$ define the shift map $\theta:W_q\to W_q$ by $\theta((\alpha_k)_{k=0}^{\infty}):=(\alpha_{k+1})_{k=0}^{\infty}$; this map is an isometry with respect to $\|\cdot\|_{W_q}$ when $1\leq q<\infty$, while it is a contraction on $W_\infty$. In the commutative setting it was shown in \cite[Proposition 2.2]{comlio} that if a weight $\alpha=(\alpha_k)_{k=0}^{\infty}\in W_\infty$ is such that $\frac{1}{n}\sum_{k=0}^{n-1}\alpha_kT^k(f)$ converges a.e. for every $f\in L_p$ and $T\in DS^+$ then the same holds if one replaces each $\alpha_k$ by $\theta(\alpha)_k=\alpha_{k+1}$. The proof uses results that are not currently available in the noncommutative setting, but the following shows that this does hold for $*$-automorphisms.

\begin{lm}\label{l32}
	Let $(\mathcal{M},\tau)$ be a semifinite von Neumann algebra,  $T:\mathcal{M}\to\mathcal{M}$ be a normal $\tau$-preserving $*$-automorphism, and $\mathcal{W}\subseteq W_1$. If $L_p(\mathcal{M},\tau)=bWW_p^{T}(\mathcal{W})$, then $L_p(\mathcal{M},\tau)=bWW_p^{T}(\bigcup_{k=0}^{\infty}\theta^k(\mathcal{W}))$.
\end{lm}
\begin{proof}
	We will first prove that $L_p(\mathcal{M},\tau)=bWW_p^{T}(\theta(\mathcal{W}))$. Assume $y\in L_p(\mathcal{M},\tau)$ and $\epsilon>0$. Let $e\in\mathcal{P}(\mathcal{M})$ be such that $\tau(e^\perp)\leq\epsilon$ and $(eM_n^\alpha(T)(T^{-1}(y))e)_{n=1}^{\infty}$ converges in $\mathcal{M}$ for every $\alpha\in\mathcal{W}$. 
	
	Fix $\alpha=(\alpha_j)_{j=0}^{\infty}\in\mathcal{W}$. Then, since
	$$
	\frac{1}{n}\sum_{j=0}^{n-1}\alpha_{j+1}T^j(y)
	=\frac{n+1}{n}\frac{1}{n+1}\sum_{j=0}^{n}\alpha_jT^j(T^{-1}(y))-\frac{1}{n}\alpha_0T^{-1}(y),
	$$
	we find that, if $T^{-1}(y)_\alpha$ is the b.a.u. limit of $(M_n^\alpha(T)(T^{-1}(y))_{n=1}^{\infty}$, then
	\begin{align*}
		\|e(M_n^{\theta(\alpha)}(T)(y)-T^{-1}(y)_\alpha)e\|_\infty
		\leq&\frac{n+1}{n}\|e(M_{n+1}^{\alpha}(T)(T^{-1}(y))-T^{-1}(y)_\alpha)e\|_\infty \\
		& \ \ \ +\frac{\|e\alpha_0T^{-1}(y)e\|_\infty}{n}\to0 \text{ as }n\to\infty.
	\end{align*}
	Since $\alpha\in\mathcal{W}$, $\epsilon>0$, and $y\in L_p$ were arbitrary, it follows that $L_p(\mathcal{M},\tau)=bWW_p^{T}(\theta(\mathcal{W}))$. By induction on $k$, it follows that $L_p(\mathcal{M},\tau)=bWW_p^{T}(\theta^k(\mathcal{W}))$ for every $k\in\mathbb{N}_0$.
	
	Fix again $y\in L_p$ and $\epsilon>0$. For every $k\in\mathbb{N}_0$ let $e_k\in\mathcal{P}(\mathcal{M})$ be such that $\tau(e_k^\perp)\leq\frac{\epsilon}{2^{k+1}}$ and $(e_kM_n^{\theta^{k}(\alpha)}(T)(y)e_k)_{n=1}^{\infty}$ converges in $\mathcal{M}$ as $n\to\infty$ for every $\alpha\in\mathcal{W}$. 
	If $e=\bigwedge_{k=0}^{\infty}e_k$, then $\tau(e^\perp)\leq\epsilon$ and $(eM_n^\alpha(T)(y)e)_{n=1}^{\infty}$ converges in $\mathcal{M}$ as $n\to\infty$ for every $\alpha\in\bigcup_{k=0}^{\infty}\theta^k(\mathcal{W})$. Hence $y\in bWW_p^{T}(\bigcup_{k=0}^{\infty}\theta^k(\mathcal{W}))$, finishing the proof.
\end{proof}

The following is the main result of this section. Its proof is inspired by that of the commutative case by \c{C}\"{o}mez and Litvinov in \cite[Theorem 2.2]{comli2}. However, the last step of their proof evaluates the functions being studied at points from an appropriate set and then use properties of the corresponding scalar sequences to finish their argument. Lacking that type of pointwise evaluation argument in the von Neumann algebra setting, we instead employ a Banach-principle style argument to circumvent this issue and finish the proof.

\begin{teo}\label{TAdmSequencesAreGood}
	Let $(\mathcal{M},\tau)$ be a semifinite von Neumann algebra, let either $p=1$ and $q=\infty$ or $1<p,q<\infty$ satisfy $\frac{1}{p}+\frac{1}{q}=1$, and let $T:\mathcal{M}\to\mathcal{M}$ be a normal $\tau$-preserving $*$-automorphism. 
	Assume that
	$\mathcal{W}\subseteq W_q$ is such that 
	$L_p(\mathcal{M},\tau)=bWW_p^{T}(\mathcal{W})$.

	If $x=(x_k)_{k=0}^{\infty}\subset L_p(\mathcal{M},\tau)^h$ is a strongly $p$-bounded $T$-admissible process, then for every $\epsilon>0$ there exists $e\in\mathcal{P}(\mathcal{M})$ with $\tau(e^\perp)\leq\epsilon$ such that the sequence
	$\left(eA_n^\alpha(x)e\right)_{n=1}^{\infty}$ converges in $\mathcal{M}$ as $n\to\infty$ for every $\alpha\in\mathcal{W}$. 
	
	Furthermore, if $\alpha\in\mathcal{W}$ and $L^\alpha(x)\in L_0(\mathcal{M},\tau)$ denotes the b.a.u. limit of $(A_n^\alpha(x))_{n=1}^{\infty}$, then it is equal to the b.a.u. limit of $(M_n^\alpha(T)(w))_{n=1}^{\infty}$, where $w$ is the exact dominant of $x$.
\end{teo}
\begin{proof}
	Since $T^k(x_0)\leq x_k$ for every $k\in\mathbb{N}_0$ by $T$-admissibility, the sequence $x^\prime:=(x_k-T^k(x_0))_{k=0}^{\infty}$ is a strongly $p$-bounded $T$-admissible process with $x_k-T^k(x_0)\geq0$ for every $k\in\mathbb{N}_0$. Since $x_0\in bWW_p^{T}(\mathcal{W})$ by assumption, the conclusion will hold for $x$ if and only if it holds for $x^\prime$. Due to this, without loss of generality we will assume that $x_k\geq0$ for every $k\in\mathbb{N}_0$.
	
	Assume $\epsilon>0$. Let $w=\sup_{k\in\mathbb{N}_0}T^{-k}(x_k)\in L_p(\mathcal{M},\tau)^+$ be the exact dominant of $x$ obtained from Lemma \ref{l31}, so that $\|w-T^{-n}(x_n)\|_p\to0$ as $n\to\infty$ and $x_n\leq T^n(w)$ for every $n\in\mathbb{N}_0$.
	
	For every $m\in\mathbb{N}_0$ write $y^{(m)}:=(y_k^{(m)})_{k=0}^{\infty}\subset L_p$, where
	$$
	y_k^{(m)}
	:=\left\{\begin{array}{cl}
		T^{k-m}(x_m),&\text{ if }k>m \\
		x_k, & \text{ if }0\leq k\leq m
	\end{array}\right.
	$$
	Observe than that $x_k-y_k^{(m)}=0$ if $0\leq k\leq m$, while if $k>m$ we have
	$$0\leq T^k(T^{-k}(x_k)-T^{-m}(x_m))=x_k-T^{k-m}(x_m)=x_k-y_k^{(m)}\leq T^k(z_m),$$ where $z_m:=w-T^{-m}(x_m).$ By the definition of $w$ we have  $$\lim_{m\to\infty}\|z_m\|_p=\lim_{m\to\infty}\|w-T^{-m}(x_m)\|_p=0.$$
	
	Due to this, since the family 
	$\left(\frac{1}{|\alpha|_{W_q}}M_n^\alpha(T)\right)_{(n,\alpha)\in\mathbb{N}\times W_q}$ is b.u.e.m. at zero on $(L_p,\|\cdot\|_p)$ by \cite[Proposition 3.1]{ob4} when $1<p<\infty$ or by \cite[Theorem 2.1]{cl1} when $p=1$ (where we adopt the convention that $\frac{0}{0}:=0$ as in \cite{ob4}), 
	for every $r\in\mathbb{N}$ there exists $m(r)\in\mathbb{N}_0$ and $f_r\in\mathcal{P(\mathcal{M})}$ such that
	$$\tau(f_{r}^\perp)\leq\frac{\epsilon}{2^{r+1}} \ \text{ and } \sup_{(n,\alpha)\in\mathbb{N}\times  W_q}\frac{\|f_{r}M_n^\alpha(T)(z_{m(r)})f_r\|_\infty}{|\alpha|_{W_q}}\leq\frac{1}{r}.$$
	
	Let $f=\bigwedge_{r=1}^{\infty}f_{r}$. Then $\tau(f^\perp)\leq\frac{\epsilon}{2}$ and
	$$\sup_{(n,\alpha)\in\mathbb{N}\times W_q}\frac{\|fM_n^\alpha(T)(z_{m(r)})f\|_\infty}{|\alpha|_{W_q}}\leq\frac{1}{r}\text{ for every }r\in\mathbb{N}.$$

	Fix $r\in\mathbb{N}$. Recall that $0\leq x_k-y_k^{(m(r))}\leq T^k(z_{m(r)})$ for every $k\geq0$. From this, if $\alpha_k\in[0,\infty)$ for every $k\in\mathbb{N}_0$ we find that
	$$0
	\leq\frac{1}{n}\sum_{k=0}^{n-1}\alpha_k(x_k-y_k^{(m(r))})
	\leq\frac{1}{n}\sum_{k=0}^{n-1}\alpha_kT^k(z_{m(r)})
	=M_n^\alpha(T)(z_{m(r)})$$

	If $\alpha=(\alpha_k)_{k=0}^{\infty}\in \mathcal{W}$, then for every $k\geq0$ we may write $\alpha_k=\sum_{s=0}^{3}i^s\alpha_{s,k}$, where $\alpha_s=(\alpha_{s,k})_{k=0}^{\infty}\in W_q^+$ and $|\alpha_s|_{W_q}\leq|\alpha|_{W_q}$ for every $s=0,1,2,3$. 
	Hence, by the above, we find that for every $n\in\mathbb{N}$,
	\begin{align}\label{e31}
		\left\|f\left(\frac{1}{n}\sum_{k=0}^{n-1}\alpha_k(x_k-y_k^{(m(r))})\right)f\right\|_\infty
		=&\left\|f\left(\frac{1}{n}\sum_{k=0}^{n-1}\left(\sum_{s=0}^{3}i^s\alpha_{s,k}\right)(x_k-y_k^{(m(r))})\right)f\right\|_\infty \nonumber\\
		\leq&\sum_{s=0}^{3}|i^s|\left\|f\left(\frac{1}{n}\sum_{k=0}^{n-1}\alpha_{s,k}(x_k-y_k^{(m(r))})\right)f\right\|_\infty \\ 
		\leq&\sum_{s=0}^{3}\left\|f\left(\frac{1}{n}\sum_{k=0}^{n-1}\alpha_{s,k}T^k(z_{m(r)})\right)f\right\|_\infty \nonumber\\
		\leq&\sum_{s=0}^{3}\frac{|\alpha_s|_{W_q}}{r}
		\leq 4\frac{|\alpha|_{W_q}}{r}. \nonumber
	\end{align}
	Hence, $\|fA_n^\alpha(x-y^{(m(r))})f\|_\infty\leq4\frac{|\alpha|_{W_q}}{r}$ for every $n,r\in\mathbb{N}$ and $\alpha\in\mathcal{W}$.

	Now, if $m\in\mathbb{N}_0$, $n>m$, and $\alpha=(\alpha_j)_{j=0}^{\infty}\in\mathcal{W}$, we find that
	$$A_n^\alpha(y^{(m)})
	=\frac{1}{n}\sum_{k=0}^{n-1}\alpha_k y_k^{(m)}
	=\frac{1}{n}\sum_{j=0}^{m-1}\alpha_kx_k+\frac{n-m}{n}\frac{1}{n-m}\sum_{k=0}^{n-m-1}\alpha_{k+m}T^k(x_m).$$
	Let $g_m\in\mathcal{P}(\mathcal{M})$ be such that $\tau(g_m^\perp)\leq\frac{\epsilon}{2^{m+1}}$, $g_{m}x_{j}g_{m}\in\mathcal{M}$ for $j=0,...,m-1$, and $(g_{m}M_{n}^{\alpha}(T)(x_m)g_{m})_{n=1}^{\infty}$ converges in $\mathcal{M}$ for every $\alpha\in\mathcal{W}$.
	
	Let $g=\bigwedge_{m=0}^{\infty}g_m$. Then $\tau(g^\perp)\leq\frac{\epsilon}{2}$ and $(gM_n^\alpha(T)(x_m)g)_{n=1}^{\infty}$ converges in $\mathcal{M}$ for every $\alpha\in\mathcal{W}$ and $m\in\mathbb{N}_0$. Since $\|g(\sum_{k=0}^{m-1}\alpha_{k}x_{k})g\|_\infty$ is constant with respect to $n$, the first term above tends to $0$ as $n\to\infty$. Therefore, we find that if $L_T^{\theta^{m}(\alpha)}(x_m)$ is the b.a.u. limit of $(M_n^{\theta^{m}(\alpha)}(T)(x_m))_{n=1}^{\infty}$ (which exists by Lemma \ref{l32} and the assumption $L_p=bWW_p^{T}(\mathcal{W})$), then we have that $\|g(A_n^\alpha(y^{(m)})-L_T^{\theta^{m}(\alpha)}(x_m))g\|_\infty\to0$ as $n\to\infty$. Since $\alpha\in\mathcal{W}$ was arbitrary, it follows that
	$(gA_n^\alpha(y^{(m)})g)_{n=1}^{\infty}$ converges in $\mathcal{M}$ for every $\alpha\in\mathcal{W}$.

	Let $e=f\wedge g$, noting that $\tau(e^\perp)\leq\epsilon$. Observe that both $f$ and $g$ are independent of a choice of $\alpha\in\mathcal{W}$, which implies so too is $e$. Furthermore, all other properties of $f$ and $g$ (with the exception of $\tau(e^\perp)\leq\epsilon$) above are also satisfied by $e$.
	
	Fix $\alpha\in\mathcal{W}$, and let $\delta>0$. Let $r\in\mathbb{N}$ be such that $\frac{4|\alpha|_{W_q}}{r}<\frac{\delta}{3}$. Note then that
	$$\sup_{n\in\mathbb{N}}\|eA_n^\alpha(x-y^{(m(r))})e\|_\infty
	\leq\sup_{n\in\mathbb{N}}\|fM_n^\alpha(T)(z_{m(r)})f\|_\infty
	\leq\frac{\delta}{3}$$
	by inequality (\ref{e31}) above. 
	Since $(eA_n^{\alpha}(y^{(m(r))})e)_{n=1}^{\infty}$ converges in $\mathcal{M}$, there exists $N\in\mathbb{N}$ such that $\ell,n\geq N$ implies
	$\|e(A_n^{\alpha}(y^{(m(r))})-A_{\ell}^{\alpha}(y^{(m(r))}))e\|_\infty
	\leq\frac{\delta}{3}.$ 
	Hence, if $\ell,n\geq N$, then
	\begin{align*}
		\|e(A_n^\alpha(x)-A_{\ell}^\alpha(x))e\|_\infty
		\leq&\|eA_n^\alpha(x-y^{(m(r))})e\|_\infty+\|eA_{\ell}^\alpha(y^{(m(r))}-x)e\|_\infty \\
		& \ \ \ 
		+\|e\big(A_n^{\alpha}(y^{(m(r))})-A_{\ell}^{\alpha}(y^{(m(r))})\big)e\|_\infty \\
		&\leq 2\frac{4|\alpha|_{W_q}}{r}+\frac{\delta}{3}
		\leq\delta.
	\end{align*}
	Since $\ell,n\geq N$ and $\delta>0$ were arbitrary, it follows that the sequence is Cauchy, implying $(eA_n^\alpha(x)e)_{n=1}^{\infty}$ converges in $\mathcal{M}$ as $n\to\infty$. Since $\alpha\in\mathcal{W}$ and $\epsilon>0$ were arbitrary, the first part of the claim has been proven.
	
	Now fix again $\alpha=(\alpha_k)_{k=0}^{\infty}\in\mathcal{W}$, and let $L_T^\alpha(x)\in L_0$ be the b.a.u. limit of $(A_n^\alpha(x))_{n=1}^{\infty}$. Since $x$ is strongly $p$-bounded, and since $\alpha\in W_1$ since $W_q\subseteq W_1$, we find that for every $n\in\mathbb{N}$,
	$$\left\|\frac{1}{n}\sum_{k=0}^{n-1}\alpha_kx_k\right\|_p
	\leq\frac{1}{n}\sum_{k=0}^{n-1}|\alpha_k|\|x_k\|_p
	\leq|\alpha|_{W_1}\sup_{j\in\mathbb{N}_0}\|x_j\|_p
	<\infty.$$
	Since the closed unit ball of $L_p$ is closed in the local measure topology of $L_0(\mathcal{M},\tau)$ by \cite[Proposition 5.14]{ddp} since $L_p$ has the Fatou property (see that article for the definitions of local measure topology and Fatou property), and since b.a.u. convergence implies convergence in measure, which itself implies convergence with respect to the local measure topology, it follows that $L_T^\alpha(x)\in L_p$.
	
	Let $L_T^\alpha(w)\in L_p$ be the b.a.u. of $(M_n^\alpha(T)(w))_{n=1}^{\infty}$. Since $T^k$ is an isometry of $L_p$ for every $k\in\mathbb{Z}$, we have that $\|x_k-T^k(w)\|_p^p=\|T^{-k}(x_k)-w\|_p^p\to0$ as $k\to\infty$ by definition of $w$, implying the Ces\`{a}ro averages of this sequence tend to $0$ as well. Hence, by H\"{o}lder's inequality if $1<q<\infty$ and by the fact that $|\alpha_j|\leq|\alpha|_{W_\infty}$ for every $j\in\mathbb{N}_0$ if $q=\infty$, we find
	$$\left\|A_n^\alpha(x)-M_n^\alpha(T)(w)\right\|_p
	\leq|\alpha|_{W_q}\left(\frac{1}{n}\sum_{k=0}^{n-1}\|x_k-T^{k}(w)\|_p^p\right)^{1/p}
	\to0 \text{ as }n\to\infty.$$ 
	Therefore
	\begin{align*}
	\|L_T^{\alpha}(x)-L_T^{\alpha}(w)\|_p
	\leq&\lim_{n\to\infty}\Big[ \|L_T^{\alpha}(x)-A_n^{\alpha}(x)\|_p+\|A_n^{\alpha}(x)-M_n^{\alpha}(T)(w)\|_p  \\
	& \ \ \ \ \ \ \ \ \ \ \ \ \ \ \ \  +\|M_n^{\alpha}(T)(x)-L_T^{\alpha}(w)\|_p\Big] \\
	=&0,
	\end{align*}
implying $\|L_T^{\alpha}(x)-L_T^{\alpha}(w)\|_p=0$ and so $L_T^{\alpha}(x)=L_T^{\alpha}(w)$.
\end{proof}

See the articles by Litvinov \cite{li2} and Hong and Sun \cite{hs} for some examples of sets of weights $\mathcal{W}\subset W_\infty$ for which this result may be applied to (namely cases (i) and (iv) from the list in the introduction, respectively); these results only hold if it is also assumed that $\tau(\textbf{1})<\infty$ and that the operator $T$ is ergodic (in the sense that if $T(x)=x$ for some $x\in L_2$, then there exists $c\in\mathbb{C}$ so that $x=c\textbf{1}$). In addition to those, applying standard approximation arguments to either one shows that if $T$ is ergodic and either $p=1$ and $q=\infty$ or $1<p,q<\infty$ satisfy $\frac{1}{p}+\frac{1}{q}=1$, then $\mathcal{W}=B_q$; the class of Besicovitch sequences from cases (ii) and (xi) from the introduction, respectively; may be used as well (the main idea can be seen in \cite[Proposition 3.6 and Corollary 3.8]{ob2}, \cite[Corollary 3.5]{ob4}, or in Theorem \ref{UnboundedIID-Duality} from Section \ref{NCWET} of this article).

If one lets $\mathcal{W}$ be a singleton set (i.e. if $\mathcal{W}=\{\alpha\}$ for some fixed $\alpha\in W_q$), then all of the weights discussed earlier in Section \ref{NCWET} (aside from the Besicovitch class from \cite{demeterjones} that we mentioned follows similar methods) may be used here as well.

One natural question that one may ask relating to the above is how far beyond the setting of $T$-admissible sequences may the methods above be generalized. In the following we will allow some error terms in the definition of $T$-admissible, but will require that the $L_p$-norms of the error terms be a summable sequence. It will turn out that any result about this new type of sequence can be found from those relating to strongly $p$-bounded $T$-admissible sequences.

\begin{pro}\label{TAdmGen}
Fix $1\leq p<\infty$. Let $x=(x_k)_{k=0}^{\infty}\subset L_p(\mathcal{M},\tau)^{+}$ be a strongly $p$-bounded sequence such that there exists another sequence $r=(r_k)_{k=0}^{\infty}\subset L_p(\mathcal{M},\tau)^{+}$ with $\sum_{j=0}^{\infty}\|r_j\|_p<\infty$ satisfying
$$T(x_k)\leq x_{k+1}+r_{k+1}$$ for every $k\in\mathbb{N}_0$. Then there exist strongly $p$-bounded $T$-admissible sequences $(y_k)_{k=0}^{\infty},(z_k)_{k=0}^{\infty}\subset L_p(\mathcal{M},\tau)^{+}$ such that
$$x_k=y_k-z_k \, \text{ for every } \, k\in\mathbb{N}_0.$$
\end{pro}
\begin{proof}
We will first show that for any increasing sequence $(\sigma_k)_{k=0}^{\infty}\subset L_p(\mathcal{M},\tau)^{+}$ that $(T^k(\sigma_k))_{k=0}^{\infty}$ is $T$-admissible.

Fix $k\in\mathbb{N}_0$. Then $\sigma_k\leq \sigma_{k+1}$, and since $T^{k+1}$ being positive implies it preserves the positive operator ordering it follows that $T(T^k(\sigma_k))=T^{k+1}(\sigma_k)\leq T^{k+1}(\sigma_{k+1})$, proving that $(T^k(\sigma_k))_{k=0}^{\infty}$ is $T$-admissible.

Now let $\sigma_k=\sum_{j=0}^{k}T^{-j}(r_j)$ for every $k\in\mathbb{N}_0$, noting that it is an increasing sequence in $L_p^+$ since each $T^{-j}(r_j)\in L_p^+$. If $z_k:=T^k(\sigma_k)$ for every $k\in\mathbb{N}_0$, then the above shows that $(z_k)_{k=0}^{\infty}$ is $T$-admissible. Notice that if $k\in\mathbb{N}_0$ is fixed, then $T$ being an isometry of $L_p$ implies that
$$\|z_k\|_p=\|T^k(\sigma_k)\|_p=\left\|\sum_{j=0}^{k}T^{-j}(r_j)\right\|_p\leq\sum_{j=0}^{k}\|T^{-j}(r_j)\|_p=\sum_{j=0}^{k}\|r_j\|_p\leq\sum_{j=0}^{\infty}\|r_j\|_p.$$ Hence
$$\sup_{k\in\mathbb{N}_0}\|z_k\|_p\leq\sum_{j=0}^{\infty}\|r_j\|_p<\infty,$$ implying that $(z_k)_{k=0}^{\infty}$ is strongly $p$-bounded.

Let $y_k=x_k+T^k(\sigma_k)$. Notice first that for any fixed $k\in\mathbb{N}_0$ that
$$\|y_k\|_p\leq\|x_k\|_p+\|T^k(\sigma_k)\|_p\leq\sup_{m\in\mathbb{N}_0}\|x_m\|_p+\sum_{j=0}^{\infty}\|r_j\|_p,$$ implying $\sup_{k\in\mathbb{N}_0}\|y_k\|_p<\infty$ and so $(y_k)_{k=0}^{\infty}$ is strongly $p$-bounded. Next, from the definition of $(x_k)_{k=0}^{\infty}$ we find that
$$T(x_k)\leq x_{k+1}+r_{k+1}=x_{k+1}+r_{k+1}+T^{k+1}(\sigma_k)-T(T^{k}(\sigma_k)),$$ implying
$$T(y_k)=T(x_k+T^k(\sigma_k))\leq x_{k+1}+r_{k+1}+T^{k+1}(\sigma_k).$$
Observe that
$$T^{-(k+1)}(r_{k+1})+\sigma_k=T^{-(k+1)}(r_{k+1})+\sum_{j=0}^{k}T^{-j}(r_j)=\sum_{j=0}^{k+1}T^{-j}(r_J)=\sigma_{k+1},$$ which implies that
$$x_{k+1}+r_{k+1}+T^{k+1}(\sigma_k)=x_{k+1}+T^{k+1}(\sigma_{k+1})=y_{k+1}.$$ Hence $T(y_k)\leq y_{k+1}$, implying that $(y_k)_{k=0}^{\infty}$ is $T$-admissible.

Since $x_k=(x_k+T^k(\sigma_k))-T^k(\sigma_k)=y_k-z_k$ for every $k\in\mathbb{N}_0$, the result has been proven.
\end{proof}

Since sums and differences of b.a.u. converge sequences also converge b.a.u., the following corollary immediately follows.
\begin{cor}
Theorem \ref{TAdmSequencesAreGood} holds for a sequence $x=(x_k)_{k=0}^{\infty}\subset L_p^+$ when it and another sequence $r=(r_k)_{k=0}^{\infty}\subset L_p^+$ satisfy the assumptions of Proposition \ref{TAdmGen}.
\end{cor}

Finally, we discuss one more generalization to the setup of Theorem \ref{TAdmSequencesAreGood} which allows any operator in $DS^+(\mathcal{M},\tau)$ to be considered instead of only $*$-automorphisms. Among many other situations one may want to consider, this would allows one to extend the result to include normal $\tau$-preserving $*$-homomorphisms $T:\mathcal{M}\to\mathcal{M}$ which are not necessarily invertible. For this, we will need a little bit more notation. We note that this new setup does not include all strongly $p$-bounded $T$-admissible processes (if one adapts the definition of $T$-admissibility in the obvious way for operators in $DS^+(\mathcal{M},\tau)$); this will be discussed more later.

A mild extra condition will be needed on the sequences that will be allowed to be used. Recall that the shift of a sequence is defined by $\theta((\alpha_k)_{k=0}^{\infty}):=(\alpha_{k+1})_{k=0}^{\infty}$.

\begin{df}
If $1\leq q\leq\infty$ and $\mathcal{W}\subseteq W_q$, then call $\mathcal{W}$ \textit{shift-invariant} if $\theta(\mathcal{W})\subseteq\mathcal{W}$, so that $\mathcal{W}=\bigcup_{k=0}^{\infty}\theta^k(\mathcal{W})$. 
\end{df}

We will discuss some relevant shift-invariant sets in $W_q$ after a few comments about the proof of Theorem \ref{TAdmDS} due to some increased context of additional requirements of them.

\begin{teo}\label{TAdmDS}
Let $(\mathcal{M},\tau)$ be a semifinite von Neumann algebra, $T\in DS^+(\mathcal{M},\tau)$, and either $1<p,q<\infty$ satisfy $\frac{1}{p}+\frac{1}{q}=1$ or $p=1$ and $q=\infty$. Assume that $\mathcal{W}\subseteq W_q$ is shift-invariant and that $L_p(\mathcal{M},\tau)=bWW_p^{T}(\mathcal{W})$.

Let $(x_k)_{k=0}^{\infty}\subset L_p(\mathcal{M},\tau)^{h}$ be an increasing strongly $p$-bounded sequence (with respect to the ordering of $L_0(\mathcal{M},\tau)^{h}$), and define $x=\sup_{k\in\mathbb{N}_0}x_k\in L_p(\mathcal{M},\tau)^{h}$ (which exists by the assumptions on the sequence). 

Then for every $\epsilon>0$ there exists $e\in\mathcal{P(M)}$ such that $\tau(e^\perp)\leq\epsilon$ and, for every $\alpha=(\alpha_k)_{k=0}^{\infty}\in\mathcal{W}$,
$$\left\|e\left(\frac{1}{n}\sum_{k=0}^{n-1}\alpha_k T^k(x_k)- L_T^{\alpha}(x)\right)e\right\|_\infty\to0 \ \text{ as } \ n\to\infty,$$ 
where $L_T^{\alpha}(x)$ denotes the b.a.u. limit of $(M_n^{\alpha}(T)(x))_{n=1}^{\infty}$.
\end{teo}

The proof is almost identical to that of Theorem \ref{TAdmSequencesAreGood} due to the assumptions made on $(x_k)_{k=0}^{\infty}$, $T$, and $\mathcal{W}$; we mention where these are used, what changes need to be made, and leave the rest of the details to the reader. 

The first main change is that one will use $T^k(x_k)$ in place of $x_k$ in the proof (so that $x_k$ itself now represents what the original proof has as $T^{-k}(x_k)$; we will mention more on this in a few paragraphs). One can still assume that $x_k\geq0$ for every $k\in\mathbb{N}_0$ by proving the result for $(x_k-x_0)_{k=0}^{\infty}$ since the fact that $(x_k)_{k=0}^{\infty}$ is increasing (so that $x_0\leq x_k\leq x_{k+1}$) implies that $0\leq x_k-x_0\leq x_{k+1}-x_0$ for every $k\in\mathbb{N}_0$ and the linearity of $T$ implies the result for $x_0$. The operator $w=\sup_{k\in\mathbb{N}_0}x_k$ would still exist for the same reasons as in Lemma \ref{l31} and would serve the same purpose. 

The fact that $T\in DS^+(\mathcal{M},\tau)$ shows that
$$\|T^k(w)-T^k(x_k)\|_p=\|T^k(w-x_k)\|_p\leq\|w-x_k\|_p\to0\text{ as }k\to\infty$$ and $T^k(x_k)\leq T^k(w)$ for every $k\in\mathbb{N}_0$. This shows why the $x_k$'s take the role of $T^{-k}(x_k)$, $z_m=w-x_m$, $y_k^{(m)}=T^k(x_m)$ for $k>m$, etc. still work with the change. 

Finally, the shift-invariance of $\mathcal{W}$ replaces the usage of Lemma \ref{l32} in the proof, since, as previously mentioned, the noncommutative version of the result \cite[Proposition 2.2]{comlio} for positive Dunford-Schwartz operators is still open.

\begin{rem} 
We will now discuss a number of shift-invariant subsets of $W_q$, where $1\leq q\leq\infty$, for which Theorem \ref{TAdmDS} applies. The list is effectively the same one that was given after Theorem \ref{TAdmSequencesAreGood}, so we will mostly justify what to do in order to fit them into this new framework. Throughout this remark, we will let either $1<p,q<\infty$ be such that $\frac{1}{p}+\frac{1}{q}=1$ or $p=1$ and $q=\infty$.

First we will introduce some extra notation. If $\alpha\in W_q$, then define $\mathcal{W}_{\alpha}:=\{\theta^k(\alpha):k\in\mathbb{N}_0\}$. Since $\theta$ is a contraction of $W_q$ (and is an isometry if $1\leq q<\infty$), it follows that $\mathcal{W}_{\alpha}\subset W_q$. Furthermore, $\mathcal{W}_{\alpha}$ is shift-invariant by definition.

One finds that $B_q$ is shift-invariant by showing that trigonometric polynomials are shift-invariant and then using the isometric (or contractive if $q=\infty$) properties of the shifting operator $\theta$ on $W_q$ (since $\|\alpha-P\|_{W_q}<\epsilon$ implies $\|\theta(\alpha)-\theta(P)\|_{W_q}<\epsilon$ for any $\alpha\in W_q$ and trigonometric polynomial $P$). From this it follows that $\mathcal{W}_{\alpha}\subset B_q$ for every $\alpha\in B_q$.

If $\mathcal{M}$ is tracial (i.e. $\tau(\textbf{1})<\infty$) and if $T$ is induced by a normal $\tau$-preserving $*$-homomorphism from $\mathcal{M}$ to itself (which need not be a $*$-automorphism), then $L_p(\mathcal{M},\tau)=bWW_p^{T}(B_q)$ by \cite[Theorem 5.1]{li2} and \cite[Remark 3.13]{ob4}.

Similarly, if $(\mathcal{M},\tau)$ is semifinite and has a separable predual and if $\alpha\in B_q$ then it follows by \cite{cls} if $p=1$ and \cite{ob4} if $1<p<\infty$ that the conclusion of Theorem \ref{TAdmDS} applies to $\mathcal{W}=\{\alpha\}$ by using the fact that $L_p(\mathcal{M},\tau)=bWW_p^{T}(\{\theta^k(\alpha)\})$ and $\theta^k(\alpha)\in B_q$ for every $k\in\mathbb{N}_0$ implies $L_p(\mathcal{M},\tau)=bWW_p^T(\mathcal{W}_{\alpha})$ (following the same reasoning as the proof of Lemma \ref{l32}) and then applying the result to $\mathcal{W}_{\alpha}$ which contains $\alpha$.

Next, one can check that if $\alpha\in W_\infty$ satisfies Inequality (\ref{ConvergenceCriteriaEquation}) in the statement of Theorem \ref{ConvergenceCriteriaI}, then similar reasoning as for why the first inequality in Remark \ref{ReIndexing} holds shows that $\theta(\alpha)$ satisfies Inequality (\ref{ConvergenceCriteriaEquation}) as well. Hence $\theta^k(\alpha)$ satisfies Inequality (\ref{ConvergenceCriteriaEquation}) for every $k\in\mathbb{N}_0$, meaning $L_p(\mathcal{M},\tau)=bWW_p^{T}(\mathcal{W}_{\alpha})$ by the same reasoning as the previous paragraph for every semifinite von Neumann algebra $(\mathcal{M},\tau)$ and $T\in DS^+(\mathcal{M},\tau)$.

Due to this, every bounded weight $\alpha$ for which Theorem \ref{ConvergenceCriteriaI} is applicable, including conditions (vi) through (x) from the Introduction, as well as special cases of (iii) and (v), may be used in Theorem \ref{TAdmDS} by applying the result to $\mathcal{W}_{\alpha}$; these sequences have no extra restrictions on the semifinite von Neumann algebra $(\mathcal{M},\tau)$ or $T\in DS^+(\mathcal{M},\tau)$. 

In the case that $1<p,q<\infty$ and $\frac{1}{p}+\frac{1}{q}=1$, for similar reasons as the above one  can find that if $(Y_k)_{k=0}^{\infty}$ is an i.i.d. sequence on a probability space $(\Omega,\mathcal{F},\mu)$ satisfying $\mathbb{E}(|Y_0|^q)<\infty$, then there exists $\Omega^{\prime\prime}\in\mathcal{F}$ with $\mu(\Omega^{\prime\prime})=1$ such that $\omega\in\Omega^{\prime\prime}$ implies that Theorem \ref{TAdmDS} applies to $\mathcal{W}=\{(Y_k(\omega))_{k=0}^{\infty}\}$.  
This is due to the fact that $(Y_{k+m})_{k=0}^{\infty}$ can be seen to also be an i.i.d. sequence on $(\Omega,\mathcal{F},\mu)$ with $\mathbb{E}(|Y_m|^q)=\mathbb{E}(|Y_0|^q)<\infty$ for every $m\in\mathbb{N}_0$. One can then take $\Omega^{\prime\prime}$ to be the intersection of the sets $\Omega_{m}^{\prime}$ obtained in Theorem \ref{UnboundedIID-Duality} for each shift $(Y_{k+m})_{k=0}^{\infty}$ of $(Y_k)_{k=0}^{\infty}$, which would be a countable intersection of full measure subsets of $\Omega$ and therefore have full measure itself.
\end{rem}

\begin{rem}
For a fixed $1\leq p<\infty$ we note that if $T\in DS^+(\mathcal{M},\tau)$ was surjective from  $L_p(\mathcal{M},\tau)$ to itself then the shift-invariance of $\mathcal{W}$ could be removed. This is due to the fact that Lemma \ref{l32} would hold for such a $T$ since surjectivity was the only special property of $T$ used in its proof.
\end{rem}




We note that while $(T^k(x_k))_{k=0}^{\infty}$ is a strongly $p$-bounded $T$-admissible for any increasing strongly $p$-bounded sequence $(x_k)_{k=0}^{\infty}\subset L_p(\mathcal{M},\tau)^{h}$, it is not necessarily the case that every $T$-admissible process is of this form unless $T$ is induced by a normal $\tau$-preserving $*$-automorphism of $\mathcal{M}$ (as was seen above). We illustrate this with a simple example.

\begin{ex}
Let $(\mathcal{M},\tau)$ be a semifinite von Neumann algebra such that there exists $e\in\mathcal{P(M)}\setminus\{0,\textbf{1}\}$ such that $\tau(e^{\perp})<\infty$. Define $T_e(x):=exe$ for every $x\in L_1(\mathcal{M},\tau)+\mathcal{M}$. It is easy to verify that $T_e\in DS^+(\mathcal{M},\tau)$. Fix $1\leq p<\infty$.

Let $(x_k)_{k=0}^{\infty}\subset L_p(\mathcal{M},\tau)^{+}$ be a strongly $p$-bounded sequence with $e^{\perp}x_je^{\perp}\neq0$ for some $j\in\mathbb{N}_0$. It follows by basic facts of the noncommutative $L_p$-norms that
$$\sup_{k\in\mathbb{N}_0}\|e^{\perp}x_ke^{\perp}\|_p\leq\|e^{\perp}\|_\infty^2\sup_{k\in\mathbb{N}_0}\|x_k\|_p<\infty,$$  
by the ordering of $L_0^+$ that $e^{\perp}x_ke^{\perp}\geq0$ for every $k\in\mathbb{N}_0$, and from the definition of $T_e$ that
$$T_e(e^{\perp}x_ke^{\perp})=ee^{\perp}x_ke^{\perp}e=0\leq e^{\perp}x_{k+1}e^{\perp}\text{ for every }k\in\mathbb{N}_0.$$ Hence $(e^{\perp}x_ke^{\perp})_{k=0}^{\infty}$ is a strongly $p$-bounded $T_e$-admissible sequence in $L_p(\mathcal{M},\tau)^{+}$.

To show that $(e^{\perp}x_ke^{\perp})_{k=0}^{\infty}$ does not directly fit in to the setup of Theorem \ref{TAdmDS}, suppose that there exists a sequence $(y_k)_{k=0}^{\infty}\subset L_p(\mathcal{M},\tau)^{h}$ such that $T_e^k(y_k)=e^{\perp}x_ke^{\perp}$ for every $k\in\mathbb{N}_0$. Then
$$e^{\perp}x_ke^{\perp}=T_e^{k}(y_k)=ey_ke=e^{2}y_ke^{2}=e(ey_ke)e=eT_e^{k}(y_k)e=ee^{\perp}x_ke^{\perp}e=0,$$ implying $e^{\perp}x_ke^{\perp}=0$ for every $k\in\mathbb{N}_0$, and in particular $e^{\perp}x_je^{\perp}=0$, which contradicts the assumption on $(x_k)_{k=0}^{\infty}$. 

Therefore the strongly $p$-bounded $T_e$-admissible sequence $(e^{\perp}x_ke^{\perp})_{k=0}^{\infty}$ cannot be found by applying powers of $T_e$ to an increasing sequence $(y_k)_{k=0}^{\infty}$ of self-adjoint operators in $L_p(\mathcal{M},\tau)^{h}$ as considered in Theorem \ref{TAdmDS}.
\end{ex}

\end{document}